\documentclass[12pt]{amsart}
\usepackage[utf8]{inputenc}

\usepackage{hyperref}
\hypersetup{nesting=true,debug=true,naturalnames=true}
\usepackage{graphicx,amssymb,upref}

\usepackage{amsmath,amsthm}
\usepackage{amsfonts}
\usepackage{latexsym}
\usepackage{amsbsy}
\usepackage{amssymb,mathscinet}

\numberwithin{equation}{section}

\usepackage{amsfonts,amssymb,amscd,amsmath,enumerate,verbatim,calc,enumerate}
\theoremstyle{plain}
\newtheorem{theorem}{Theorem}[section]
\newtheorem{corollary}[theorem]{Corollary}
\newtheorem{lemma}[theorem]{Lemma}
\newtheorem{proposition}[theorem]{Proposition}

\newtheorem{notation}[theorem]{Notation}
\newtheorem{definition}[theorem]{Definition}

\newtheorem{remark}[theorem]{Remark}

\newcommand{\ot}{\otimes}

\newcommand{\wt}{\widetilde}
\newcommand{\wV}{\wt{V}}

\DeclareMathOperator{\dis}{dist}

\begin{document}

\title[ Regular covariant representations and their Wold-type decomposition ]
{Regular covariant representations and their Wold-type decomposition}

\date{\today}
\author[Rohilla]{Azad Rohilla}
\address{Centre for Mathematical and Financial Computing, Department of Mathematics, The LNM Institute of Information Technology, Rupa ki Nangal, Post-Sumel, Via-Jamdoli
	Jaipur-302031,
	(Rajasthan) INDIA}
\email{18pmt005@lnmiit.ac.in}
\author[Trivedi]{Harsh Trivedi}
\address{Centre for Mathematical and Financial Computing, Department of Mathematics, The LNM Institute of Information Technology, Rupa ki Nangal, Post-Sumel, Via-Jamdoli
	Jaipur-302031,
	(Rajasthan) INDIA}
\email{harsh.trivedi@lnmiit.ac.in, trivediharsh26@gmail.com}
	\author[Veerabathiran]{Shankar Veerabathiran \textsuperscript{*}}
\address{Indian Statistical Institute, Statistics and Mathematics Unit, 8th Mile, Mysore Road,
	Bangalore, 560059, India}
\email{shankarunom@gmail.com}
\thanks{*corresponding author}


\begin{abstract}
	Olofsson introduced a growth condition regarding elements of an orbit for an expansive operator and generalized Richter's wandering subspace theorem. Later on, using the Moore-Penrose inverse, Ezzahraoui, Mbekhta, and Zerouali extended the growth condition and obtained a Shimorin-Wold-type decomposition. Shimorin-Wold-type decomposition for completely bounded covariant representations, which are close to isometric representations, is obtained in \cite{HV19}. This paper extends this decomposition for regular, completely bounded covariant representation having reduced minimum modulus $\geq 1$  that satisfies the growth condition. To prove the decomposition, we introduce the terms regular, algebraic core, and reduced minimum modulus in the completely bounded covariant representation setting and work out several fundamental results. Consequently, we shall analyze the weighted unilateral shift introduced by Muhly and Solel and introduce and explore a non-commutative weighted bilateral shift.
\end{abstract}

\keywords{Covariant representations, Hilbert $C^*$-modules,	wandering subspaces, Wold decomposition, Algebraic core, Moore-Penrose inverse, regular operator and reduced minimum modulus.}
\subjclass[2010]{  47A15, 47L30, 46L08,  47L55, 47L80, 47B38.}

\maketitle
\section{Introduction}

The classical result of Wold \cite{W}  asserts that given isometry on a Hilbert space is either a unitary, a shift, or  uniquely breaks as a direct summand of them. Beurling \cite{BA} proved that every $M_{z}$-invariant closed subspace of the Hardy space $ H^{2}(\mathbb{D})$ is a copy of an inner function. One of the well-known implications of the Wold decomposition is that when the  unitary part is zero, the Wold-decomposition gives uniqueness of the wandering subspace for a shift.  Halmos {\cite{H61}} proved a wandering subspace theorem, which is an abstraction of the Beurling's theorem, that characterized all the invariant subspaces of a shift. Richter \cite{R88} proved a wandering subspace theorem for an analytic concave operator which satisfies the growth condition that was explicitly mentioned and generalized by Olofsson in {\cite{O05}}. After that, Shimorin \cite{S01} provided an elementary proof of Richter's theorem by giving a Wold-type decomposition for concave operators that can be considered close to an isometry. Olofsson {\cite{O05}} extended the Richter's wandering subspace theorem as follows:
\begin{theorem}(Olofsson)\label{Oa}
	Suppose $V$ is an analytic bounded linear map defined on Hilbert space $\mathcal{H}$ such that \begin{itemize}
		\item [(i)] $V$ is an expansive operator,
		\item [(ii)] there exist some positive numbers $d_{m} , d $  such that $\sum_{m \geq 2 } \frac{1}{d_{k}} =\infty$ and \begin{equation*}
		\|V^{m}h\|^{2}\leq d\|h\|^{2}+ d_{m}(\|Vh\|^{2}-\|h\|^{2})  , \quad h \in \mathcal{H}.
		\end{equation*} Then $V$ has the wandering subspace property.
	\end{itemize}
\end{theorem}

Ezzahraoui, Mbekhta, and Zerouali in \cite{EMZ15} using the reduced minimum modulus $\geq 1,$ and a general condition of the expansive operator, extended Theorem \ref{Oa}  for the broader class of regular operators \cite{M89}  and proved the following Wold-type decomposition: \begin{theorem}(Ezzahraoui-Mbekhta-Zerouali)
	If $V$ is a  regular bounded linear operator defined  on a Hilbert space $\mathcal{H}$ such that the reduced minimum modulus of $V$ is greater than or equal to 1, and  satisfies the  growth condition \begin{align*}
	\|V^{m}h\|^{2}\leq  \|V^{\dagger}Vh\|^{2} +d_{m}(\|Vh\|^{2}-\|V^{\dagger}Vh\|^{2}),\:\:\: h \in \mathcal{H}
	\end{align*} such that $\sum_{ m\geq 2}\frac{1}{d_{m}},$ then \begin{equation*}
	\mathcal{H}= \bigcap_{n=1}^{\infty}V^{n}(\mathcal{H})+[\mathcal{H}\ominus V(\mathcal{H})]_{V} .
	\end{equation*}
\end{theorem}

The study of Wold decomposition has begun in Non-commutative Multivariable Operator Theory  by Frazho \cite{F82} and by Popescu \cite{Po89}. They explored Wold-decomposition for row-isometries considered by Cuntz \cite{C77}. Pimsner \cite{P97} introduced the notion of Cuntz-Pimsner algebra for faithful $C^*$-correspondence. Muhly and Solel \cite{MS99} provided the Wold-type decomposition for representations of tensor algebras of $C^*$-correspondences, and they explored the invariant subspaces of particular subalgebras of Cuntz-Krieger algebras.

Let us discuss the structure of this article: Section \ref{section 3} begins with the introduction of generalized range, the algebraic core, and regular covariant representation. Here, we explain the connection between generalized range and algebraic core. In Section \ref{section 4}, we analyze the regularity condition of the generalized inverse of a covariant representation, discuss its properties, and characterize the element of the generalized range. In Section \ref{section 5},  we study the reduced minimum modulus, Moore-Penrose inverse for the covariant representation $(\sigma, V)$ and its various properties. Assume $(\sigma,V)$ to be regular, completely bounded, covariant representation of a $C^*$-correspondence $E$ on a Hilbert space $\mathcal{H}$  with the reduced minimum modulus $\gamma(\widetilde{V}) \geq 1$ and satisfying the growth condition \begin{equation*}\label{concave}
\|\widetilde{V}_k (\xi_k )\|^2 \leq d_k (\| (I_{E^{\ot {k-1}}} \ot \widetilde{V}) (\xi_k)\|^2 -\|(I_{E^{\ot {k-1}}} \ot \widetilde{V}^\dagger \widetilde{V}) (\xi_k)\|^2) + \|(I_{E^{\ot {k-1}}} \ot \widetilde{V}^\dagger \widetilde{V}) (\xi_k)\|^2
\end{equation*} for every $ k\geq 0 $;\quad  $ \xi_k \in E^{\ot k} \ot \mathcal{H}$ with $\sum_{k \geq 2}$ $\frac{1}{d_k} = \infty$. Then $(\sigma,V)$ admits a Shimorin-Wold-type decomposition. It is the main content of \ref{section 6}.  In Section \ref{section 7}, we extend the commutant result of Bercovici, Douglas, and  Foias \cite{BRC} for a shift to left invertible completely bounded representation based on recent work of S. Sarkar\cite{SS21}. In Section \ref{7}, we derived a sufficient condition on weight sequence $(Z_{k} )_{k \in \mathbb{N}_{0}}$  such that weighted shift $(\rho, S)$ on a $F(E)\ot \mathcal{H}$ with weight sequence $(Z_{k})$ given by Muhly and Solel in \cite{MS16} admits the generating wandering subspace property. We also study the multivariable analog of the bilateral shift discussed in \cite{EMZ15}. 
\subsection{ Preliminaries and Notations}
Now we recall some well-known definitions and properties
of Hilbert $C^*$-modules, $C^*$-correspondences (cf. \cite{MR0355613, La95}) and  covariant representations of $C^*$-correspondences (cf. \cite{P97, MR1648483}).

Let $\mathcal{B}$ be a $C^*$-algebra and  $E$ be a Hilbert $\mathcal{B}$-module. Assume $\mathcal L(E)$ to be $C^*$-algebra of all adjointable maps on $E.$ Now $E$  is said to be a {\em $C^*$-correspondence over $\mathcal B$} if $E$ is left $\mathcal
B$ a module where the left action is through a non-zero $*$-homomorphism
$\phi:\mathcal B\to \mathcal L(E)$ such that
\[a\xi:=\phi(a)\xi \quad \mbox{for all} \quad (a\in\mathcal B, \xi\in E).\] Here, we assume that every $*$-homomorphism is always
essential, which means that the closure of the linear span of
$\phi(\mathcal B)E$ is $E.$ The module $E$ has an operator
space structure inherited as a subspace of the so called linking algebra (see \cite{MR1648483}). Suppose $F$ and $E$ are two
$C^*$-correspondences over $\mathcal B$. Then the
{\em tensor product}, denoted by $F\otimes_{\phi} E$, satisfies the following properties
\[
(\xi_1 a)\otimes \zeta_1=\xi_1\otimes \phi(a)\zeta_1,
\]
\[
\langle\xi_1\otimes\zeta_1,\xi_2\otimes\zeta_2\rangle=\langle\zeta_1,\phi(\langle\xi_1,\xi_2\rangle)\zeta_2\rangle
\]
for every $\xi_1,\xi_2\in F,$ $\zeta_1,\zeta_2\in E$  and $a\in\mathcal{B}.$

Unless necessary, after this section, we simply write $F\otimes E$    instead of  $F\otimes_{\phi} E.$  Throughout this paper, we assume that $\mathcal{B}$ is a $C^*$-algebra, $\mathcal{H}$ is a Hilbert
space and  $E$ is a $C^*$-correspondence over $\mathcal B$ with left module action given by a $*$-homomorphism $\phi:\mathcal{B} \rightarrow \mathcal{L}(E).$
\begin{definition}
	Assume
	$\sigma:\mathcal B\to B(\mathcal H)$ to be a representation and $V:
	E\to B(\mathcal H)$ to be a linear function.  If
	\[
	V( c\eta d)=\sigma(c)V(\eta)\sigma(d), \quad \mbox{for all}\quad
	c,d\in\mathcal B,\eta\in E
	\] then we say that the pair $(\sigma,V)$ is
	a {\em covariant representation} (cf. \cite{MR1648483}) of $E$ on the Hilbert space $\mathcal{H}.$ The covariant representation $(\sigma, V)$ is {\em completely bounded}  (CB-representation)  if $V$ is { \em completely bounded}. Additionally, if $V(\xi)^*V(\zeta)=\sigma(\langle \xi, \zeta \rangle), \, \zeta, \xi \in E$  then  we say $(\sigma,V)$ is {\em isometric.} 
\end{definition}

Consider a CB-representation  $(\sigma, V)$ of $E$ on $\mathcal{H},$  then there corresponds  a bounded operator   $\widetilde{V}: E\otimes \mathcal{H} \to \mathcal{H}$ defined by $\widetilde{V}(\xi \otimes h):=V(\xi)h$ where $ \xi \in E$ and $ h\in\mathcal{H}.$
Note that $\wt{V}$ satisfies $\wt{V}(\phi(a) \ot I_{\mathcal{H}})=\sigma(a)\wt{V}, a \in \mathcal{B}$ and $\phi$ is a left action  on $E.$

To categorise the covariant representation of $C^*$-correspondences, the following lemma from {\rm \cite[Lemma 3.5]{MR1648483}} is useful. \begin{lemma}\label{MSL} 
	There is a map $(\sigma, V)\mapsto \widetilde{V} $ which gives a  one-to-one correspondence between the collection of all bounded  linear maps
	$\widetilde{V}:~\mbox{$E\otimes_{\sigma} \mathcal{H}\to \mathcal{H}
		$}$  which satisfies $\widetilde{V}(\phi(a)\otimes I_{\mathcal
		H})=\sigma(a)\widetilde{V},$ $a\in\mathcal B$ and the collection of all CB-representations $(\sigma, V)$ of $E$ on $\mathcal H.$ Additionally, $(\sigma, V)$ is isometric if and only $\widetilde{V}$ is an isometry. A CB-representation $(\sigma, V)$ is {\em fully co-isometric} if $\widetilde{V}$ is co-isometry, that is, $\widetilde{V}\widetilde{V}^*=I_{\mathcal{H}}.$
\end{lemma}



For each $n \in \mathbb{N},$ define $E^{\otimes n}:=E \otimes \cdots \otimes E=\displaystyle\bigotimes_{i=1}^{n}E$ ($n$ fold tensor product) and $E^{\otimes 0}:=\mathcal{B}.$ Then $E^{\otimes n}$ is  a $C^*$-correspondence over $\mathcal{B}$ in a natural way,  where the left action  of $\mathcal{B}$ on $E^{\otimes n}$   is denoted by $\phi^n$  and, is given by $\phi^n(a)(\bigotimes_{i=1}^{n}\xi_i)=\phi(a)\xi_{1}\ot\bigotimes_{i=2}^{n}\xi_i, \: \xi_i \in E, 1 \leq i \leq n.$

For each $n \in \mathbb{N},$ define  $\widetilde{V}_n:~\mbox{$E^{\otimes n}\otimes \mathcal H\to \mathcal
	H$}$  by  \[
\widetilde{V}_n(\bigotimes_{i=1}^{n}\xi_i \otimes h)=V(\xi_1) V( \xi_2) \cdots V(\xi_n)h,
\]
for $\xi_i\in E,
h\in\mathcal H, 1 \leq i \leq n.$ It is trivial to see that for each $ n\in \mathbb{N},$
$$\widetilde{V}_n=\widetilde{V}(I_E \otimes \widetilde{V}_{n-1} )= \wV_{n-1}(I_{E^{\otimes n-1}} \otimes \widetilde{V}).
$$

The following theorem from \cite[Theorem 3.13]{HV19} is an abstraction of Shimorin-Wold-type decomposition \cite[Theorem 3.6]{S01} and of Wold decomposition for isometric representation due to Muhly and Solel \cite{MS99}.

\begin{theorem}\label{MT1}
	Consider a CB-representation $(\sigma, V)$  of $E$ on  $\mathcal{H}.$ If it satisfies one of the following conditions:
	\begin{enumerate}
		
		\item[(1)]  for every  $\zeta \in E^{\otimes 2} \otimes \mathcal{H},\kappa \in E \otimes \mathcal{H} $
		\begin{align*}
		\|(I_{E} \otimes \wt{V})(\zeta)+\kappa\|^2 \leq 2(\|\zeta\|^2+\|\wt{V}(\kappa)\|^2);
		\end{align*}
		\item[(2)]  $(\sigma, V)$ is concave, i.e., \\ $\| \wt{V}_2(\eta \otimes h)\|^2+\|\eta \otimes h\|^2 \leq 2 \|(I_E \otimes  \wt{V})(\eta \otimes h)\|^2, \eta \in E^{\otimes 2}, h \in \mathcal{H}.$
	\end{enumerate}
	Then $(\sigma, V)$ admits Wold-type decomposition.  To put it another way, there is a wandering subspace $\mathcal{W}$ for $(\sigma, V)$ that divides $\mathcal{H}$ into the direct sum of $(\sigma, V)$-reducing subspaces.
	\begin{align}
	\mathcal{H}=\displaystyle\bigvee_{n \in \mathbb{N}_0}\wt{V}_n(E^{\ot n} \ot \mathcal{W}) \bigoplus \bigcap_{n \in \mathbb{N}_{0}}\wt{V}_n(E^{\otimes n} \otimes \mathcal{H})
	\end{align}
	that both an isometric and a fully co-isometric representation of the restriction of $(\sigma,V)$ on $\bigcap_{n \in \mathbb{N}}\wt{V}_n(E^{\otimes n} \otimes \mathcal{H})$ is achieved. Moreover, the above decomposition is unique.
\end{theorem}
\section{Regular, completely bounded, covariant representations}\label{section 3} This section is based on \cite[Chapter 1]{A07}. Also, we introduce regular covariant representations of ${E}$ on $\mathcal{H}$ and prove that its generalized range (hyper-range) is equal to the notion of the algebraic core (cf. \cite[Chapter 1]{A07}).

Suppose $(\sigma,V)$ is a CB-representation of  $E$ on $\mathcal{H}.$  Define $ R^\infty{({V})}:= \bigcap_{n \in \mathbb{N}_{0}}R(\widetilde{V}_{n} )$ will be used to denote by  the generalized range  of  $(\sigma,V),$ where $\mathbb{N}_{0}:=\mathbb{N}\cup\{0\}.$ Since $(R(\widetilde{V}_{n}))_{n\in \mathbb{N}_{0}}$ is a decreasing sequence,  $ R^\infty{({V})}= \bigcap_{n \in \mathbb{N}_{0}}R(\widetilde{V}_{n})=\bigcap_{n \in \mathbb{N}}R(\widetilde{V}_{n} ).$ 

\begin{lemma} \label{LL}
	Consider a CB-representation $(\sigma,V)$  of ${E}$ on $\mathcal{H}.$ Then  \begin{equation*}
	(I_{E^{\otimes n}} \otimes \widetilde{V}_m) N(\widetilde{V}_{m+n}) = N(\widetilde{V}_{n})  \cap (I_{E^{\otimes n}} \otimes \widetilde{V}_m) (E^{\otimes (m+n)} \otimes \mathcal{H}), \quad \quad m,n \in \mathbb{N}.
	\end{equation*}
\end{lemma}	
\begin{proof}
	Let $ \xi \in N(\widetilde{V}_{m+n}) \subseteq E^{\otimes (m+n)} \otimes \mathcal{H},$ then $ \widetilde{V}_{n}(I_{E^{\otimes n}} \otimes \widetilde{V}_m)(\xi) = \widetilde{V}_{m+n}(\xi)=0 $, so that $(I_{E^{\otimes n}} \otimes \widetilde{V}_m) N(\widetilde{V}_{m+n}) \subseteq N(\widetilde{V}_{n})  \cap (I_{E^{\otimes n}} \otimes \widetilde{V}_m) (E^{\otimes (m+n)} \otimes \mathcal{H})$.
	
	On another hand, if $ \eta \in N(\widetilde{V}_{n})  \cap (I_{E^{\otimes n}} \otimes \widetilde{V}_m) (E^{\otimes (m+n)} \otimes \mathcal{H})$ then $\widetilde{V}_{n}(\eta)= 0 $ and $ \eta = (I_{E^{\otimes n}} \otimes \widetilde{V}_m)(\xi)$ for some $\xi \in E^{\otimes (m+n)} \otimes \mathcal{H}$, and we get  $\widetilde{V}_{m+n}(\xi)=\widetilde{V}_{n}(I_{E^{\otimes n}} \otimes \widetilde{V}_m)(\xi) = 0 .$ Therefore $  \xi \in N(\widetilde{V}_{m+n}),$ we obtain $\eta \in(I_{E^{\otimes n}} \otimes \widetilde{V}_m) N(\widetilde{V}_{m+n}),$ so the opposite inclusion is verified.\qedhere
	
\end{proof}
The  following result defines the notion of regular, which is generalizing \cite[Theorem 1.5]{A07}.
\begin{theorem}\label{cc}
	Consider a CB-representation $(\sigma,V)$  of  $E$ on $\mathcal{H}.$ Then  the following conditions are equivalent:
	\begin{enumerate}
		\item $ N(\widetilde{V}) \subseteq (I_{E} \otimes \widetilde{V}_n)(E^{\otimes(n+1) }\otimes \mathcal{H})$  for every  $ n \in \mathbb{N} $ ;
		\item $ N({\widetilde{V}}_n) \subseteq (I_{E^{\otimes n}} \otimes \widetilde{V})(E^{\otimes(n+1) }\otimes \mathcal{H})$  for every  $n \in \mathbb{N} $ ;
		\item $ N({\widetilde{V}}_n) \subseteq(I_{E^{\otimes n}} \otimes \widetilde{V}_m)(E^{\otimes(n+m) }\otimes \mathcal{H})$   for every  $n, m \in \mathbb{N} $ ;
		\item $N(\widetilde{V}_{n}) = (I_{E^{\otimes n}} \otimes \widetilde{V}_m)N (\widetilde{V}_{m+n})$  for every  $n, m \in \mathbb{N} .$
	\end{enumerate}
	
\end{theorem}
\begin{proof}
	$(1) \implies (2):$ We will prove it by Mathematical induction.
	For $n=1$ in (1), we have \begin{equation*}
	N(\widetilde{V}) \subseteq (I_{E} \otimes \widetilde{V})(E^{\otimes2 }\otimes \mathcal{H}),
	\end{equation*} which means (2) is true for $ n=1.$
	Let us assume (2) is true for $ n=k$. To show that it is true for $ n=k+1$, i.e., \begin{equation*}
	N({\widetilde{V}}_{k+1}) \subseteq (I_{E^{\otimes {k+1}}} \otimes \widetilde{V})(E^{\otimes(k+2) }\otimes \mathcal{H}).
	\end{equation*}
	For this purpose, consider $\xi \in N{(\wV_{k+1})}$ then
	\begin{align*}
	(I_{E}\ot \wV_{k})(\xi)\in N{(\wV)}\subseteq (I_{E} \otimes \widetilde{V}_{k+1})(E^{\otimes(k+2) }\otimes \mathcal{H}),
	\end{align*} here last inequality follows from the hypothesis  (1). By Lemma \ref{LL},  there  exists $ \eta \in N{(\wV_{k+2})}$ such that $ (I_{E} \ot \wV_{k})(\xi)=(I_{E}\ot\wV_{k+1} )(\eta)$ which implies that
	$\xi-(I_{E^{\ot k+1}}\ot \wV)\eta\in N(I_{E} \ot \wV_{k})=E\ot N(\wV_{k})$. Now using the induction hypothesis, i.e., $ N({\widetilde{V}}_{k}) \subseteq (I_{E^{\otimes {k}}} \otimes \widetilde{V})(E^{\otimes(k+1) }\otimes \mathcal{H})$, we get \begin{equation*}
	\xi \in N({\widetilde{V}}_{k+1}) \subseteq (I_{E^{\otimes {k+1}}} \otimes \widetilde{V})(E^{\otimes(k+2) }\otimes \mathcal{H}),
	\end{equation*} which proves the desired inequality.

	$(2) \implies(3):$ We will prove inequality (3) by Mathematical induction. For $ m=1 $, nothing to prove. Suppose (3) is true for $ m =k.$ Now we need to show that it is correct for $m =k+1,$ that is,\begin{equation*}
	N({\widetilde{V}}_n) \subseteq(I_{E^{\otimes n}} \otimes \widetilde{V}_{k+1})(E^{\otimes(n+k+1) }\otimes \mathcal{H}).
	\end{equation*} Let $\xi\in N({\widetilde{V}}_n) ,$ using (2) there exists $ \eta\in N{(\wV_{n+1})} $ such that $\xi = (I_{E^{\ot n}} \ot \wV)\eta$ and by induction hypothesis, we have \begin{align*}
	\xi = (I_{E^{\ot n}} \ot \wV)\eta \in(I_{E^{\ot n}} \ot \wV) (I_{E^{\ot n+1}} \ot \wV_{k})(E^{\ot {n+k+1}}\ot \mathcal{H})=(I_{E^{\otimes n}} \otimes \widetilde{V}_{k+1})(E^{\otimes(n+k+1) }\otimes \mathcal{H}).
	\end{align*} Thus we get (3).\\
	$(3)\implies (4):$ Assume $ N({\widetilde{V}}_n) \subseteq(I_{E^{\otimes n}} \otimes \widetilde{V}_m)(E^{\otimes(n+m) }\otimes \mathcal{H})$, for $ m, n \in \mathbb{N}$. By Lemma \ref{LL} and (3), \begin{align*}
	(I_{E^{\otimes n}} \otimes \widetilde{V}_m) N(\widetilde{V}_{m+n}) = N(\widetilde{V}_{n})  \cap (I_{E^{\otimes n}} \otimes \widetilde{V}_m) (E^{\otimes (m+n)} \otimes \mathcal{H})= N({\wV_{n}})
	\end{align*} which proves (4).

	$(4) \implies(1):$ Obvious.
\end{proof}

Note that if $\wV$ satisfy (3) of the above theorem, we get \begin{equation*}
N({\widetilde{V}}_n) \subseteq E^{\otimes n} \otimes R^{\infty}({V}),\:\:\:n \in \mathbb{N}.
\end{equation*} In particular $N({\widetilde{V}}) \subseteq E \otimes R^{\infty}({V}).$ On another hand, if $N({\widetilde{V}}) \subseteq E \otimes R^{\infty}({V}),$ then $\wV$ satisfies all conditions of above theorem.

\begin{definition}
	Consider  a CB-representation  $(\sigma,V)$  of  $E$ on $\mathcal{H}.$ We call $(\sigma,V)$ is {\em regular} if range of $\widetilde{V}$ is closed and it satisfy any one of the equivalent conditions of Theorem \ref{cc}.
\end{definition}
\begin{definition}[Algebraic Core of  $(\sigma,V)$]
	
	Consider  a CB-representation $(\sigma,V)$ of $E$ on $\mathcal{H}.$ The {\rm algebraic core} of $(\sigma,V)$, denote by $C(\wV)$, is the greatest invariant subspace $\mathcal{K}$ of $\mathcal{H}$ such that $\widetilde{V}(E \ot \mathcal{K} ) =\mathcal{K} .$
\end{definition} Clearly for all CB-representation $(\sigma,V)$ of $E$ on $\mathcal{H}$, we have \begin{align*}
C(\widetilde{V}) = \widetilde{V}(E \ot C(\widetilde{V})) &=\widetilde{V}(E \ot \widetilde{V}(E \ot C(\widetilde{V})))  =\widetilde{V}({I_E \ot \widetilde{V}})(E \ot E \ot C(\widetilde{V})) \\& = \widetilde{V}_2(E^{\ot 2} \ot C({\widetilde{V}}))= \cdots = \widetilde{V}_n(E^{\ot n} \ot C({\widetilde{V}})) \subseteq \widetilde{V}_n(E^{\ot n} \ot \mathcal{H}), \quad  n \in \mathbb{N}.\end{align*}It follows that $C(\widetilde{V}) \subseteq R^{\infty}({V})$.

The following theorem gives us an analytic approach to the algebraic core of  $(\sigma,V)$.
\begin{theorem}\label{Core}
	Consider a CB-representation $(\sigma,V)$  of $E$ on $\mathcal{H}.$ If $ h \in C(\widetilde{V}) $ then there is a sequence $(\xi_n) $ with $ h= \xi_0 $
	and  $\xi_n\in  {E}^{\otimes n} \ot  C(\widetilde{V}),$ such that $(I_{{E}^{\otimes n}}\ot \widetilde{V})( \xi_{n+1})=\xi_n.$
\end{theorem}
\begin{proof} Let $ h \in C(\widetilde{V}),$
	by definition of $C(\wV)$ there is an element $\xi_1  \in E \ot C(\widetilde{V})$ such that $ h = \widetilde{V}( \xi_1)$. 	
	Since $ \xi_1 \in E \ot C(\widetilde{V})=(I_E\ot \widetilde{V})({E}^{\otimes 2} \ot  C(\widetilde{V})),$ there is $\xi_2\in {E}^{\otimes 2} \ot  C(\widetilde{V})$ such that $(I_E\ot \widetilde{V})(\xi_2)=\xi_1,$ and so $ h=\widetilde{V}( \xi_1)=\widetilde{V} (I_E\ot \widetilde{V})(\xi_2)=\widetilde{V}_2(\xi_2).$ More generally, there exists a sequence $(\xi_n)$, with $\xi_n\in  {E}^{\otimes n} \ot  C(\widetilde{V}),$ for which $ h= \xi_0 $
	and $(I_{{E}^{\otimes n}}\ot \widetilde{V})( \xi_{n+1})=\xi_n.$
\end{proof}
The following theorem establishes the relationship between the generalized range and the algebraic core of $(\sigma,V)$.
\begin{theorem}\label{reg}
	Consider  a regular CB-representation  $(\sigma,V)$ of $E$ on $\mathcal{H}.$ Then \begin{equation*}
	\widetilde{V}({E} \otimes{ {{R}^{\infty}{({V})}}}) = { {{R}^{\infty}{({V})}}}=C(\widetilde{V}).
	\end{equation*}
\end{theorem}
\begin{proof} For $ h\in  {R}^{\infty}{({V})} \:\big(=\bigcap_{n\in \mathbb{N}_{0}}R(\widetilde{V}_{n})\big),$ there is $\eta_n\in {E}^{\otimes n}\otimes\mathcal{H}$ such that	$h = {\widetilde{V}}_n(\eta_n).
	$
	Now apply an operator $V(\xi)$ both sides \begin{align*}
	V(\xi)h=V(\xi)  {\widetilde{V}}_n(\eta_n)
	= \widetilde{V}_{n+1}(\xi\otimes \eta_n)  \in R(\widetilde{V}_{n+1}), \quad    n\in \mathbb{N},
	\end{align*}  where  $ \xi \in E$.
	Since $	\widetilde{V}(\xi\otimes h)$ $\in$ ${ {{R}{({\widetilde{V}}_n)}}}$ for every $ n\in \mathbb{N},$  $	\widetilde{V}(\xi\otimes h)$ $\in$ ${ {{R}^{\infty}{(\widetilde{V})}}}.$
	Thus, we obtain
	\begin{equation*}
	\widetilde{V}({E}\otimes{ {{R}^{\infty}{({V})}}}) \subseteq{ {{R}^{\infty}{({V})}}}.
	\end{equation*}
	
	For the	converse part, let $h\in { {{R}^{\infty}{({V})}}},$ then there is a sequence ($\eta_n$), where $\eta_n\in {E}^{\otimes n}\otimes\mathcal{H},$ so that \begin{equation*}
	h =\widetilde{V}(\eta_1) = \widetilde{V}_2(\eta_2) = \cdots = \widetilde{V}_{n+1}(\eta_{n+1})=\widetilde{V}(I_E \otimes \widetilde{V}_{n})(\eta_{n+1})
	\end{equation*} and hence 		
	$\eta_1 - (I_E\otimes{\widetilde{V}}_n)(\eta_{n+1}) \in { {{N}{(\widetilde{V})}}}.$ Since $(\sigma,V)$ is  regular,
	\begin{equation*}
	\eta_1 - (I_E\otimes{\widetilde{V}}_n)(\eta_{n+1})  \in{E}  \otimes { {{R}^{\infty}{({V})}}}.
	\end{equation*}	Since ${E}\otimes{ {{R}^{\infty}{({V})}}}$ is a subspace,  we get $\eta_1 \in {E}\otimes{ {{R}^{\infty}{({V})}}}.$   Thus, $ h \in \widetilde{V}({E} \otimes{ {{R}^{\infty}{({V})}}}$) which gives  ${ {{R}^{\infty}{({V})}}}\subseteq\widetilde{V}(E \otimes{ {{R}^{\infty}{({V})}}}),$ then  $\widetilde{V}({E} \otimes{ {{R}^{\infty}{({V})}}}) = { {{R}^{\infty}{({V})}}}$  and  by the definition of $C(\widetilde{V})$, we deduce that $R^{\infty}({V})\subseteq C(\widetilde{V}).$ Since $C(\widetilde{V}) \subseteq R^{\infty}({V}),$ we  get  $R^{\infty}({V})= C(\widetilde{V})  $ which completes the proof of the theorem.
\end{proof}

\section{Regularity condition of the generalized inverse of a covariant representation}\label{section 4}  Definitions of the generalized inverse and bi-regularity of the covariant representation are provided in this section and derive some basic properties.  Further, we prove the generalized range of a covariant representation is invariant under the generalized inverse.
\begin{definition}
	Consider a CB-representation $(\sigma,V)$ of $E$ on $\mathcal{H}.$ A bounded operator $S: \mathcal{H} \rightarrow E\ot \mathcal{H}$  is said to be a {\em generalized inverse} of $\widetilde{V}$ if 	$
	\widetilde{V}S\widetilde{V}=\widetilde{V}$ and $ S\widetilde{V}S=S.
	$
\end{definition}
Assume $(\sigma,V)$ to be a CB-representation of $E$ on $\mathcal{H}$ and let $S$ be a generalized inverse of $\widetilde{V}$. For $n \in \mathbb{N} $, define
$S^{(n)}: H \rightarrow E^{\ot n } \ot \mathcal{H}$ by \begin{align}\label{snn}
S^{(n)}:= (I_{E^{\otimes n-1}}\ot S)(I_{E^{\otimes n-2}}\ot S) \cdots(I_{E}\ot S)S.
\end{align}
Observe that  \begin{align*}
(I_{E^{\ot m}} \ot S^{(n)})S^{(m)}=S^{(m+n)},\:\:\:  m,n \in \mathbb{N}.\qedhere
\end{align*} Note that if  $(\sigma,V)$ is  isometric, then $S = \widetilde{V}^{*}$.

The next lemma will take a critical role in defining the bi-regularity condition of the CB-representation.

\begin{lemma}\label{kk}
	Consider a CB-representation $(\sigma,V)$  of $E$ on $\mathcal{H}.$ Suppose $S$ is a generalized inverse of $\widetilde{V}$. Then for every $ m,n \in \mathbb{N},$
	\begin{equation*}	S^{(m)} N(S^{(m+n)}) =N(I_{E^{\ot m}} \ot  S^{(n)})  \cap S^{(m)}( \mathcal{H}).
	\end{equation*}
\end{lemma}	
\begin{proof}
	Let $n, m \in \mathbb{N}$ and  $ h \in N(S^{(m+n)}),$ then  $
	(I_{E^{\ot m}} \ot S^{(n)})S^{(m)}h =S^{(m+n)}h = 0
	$ and thus $S^{(m)}h\in N(I_{E^{\ot m}} \ot  S^{(n)}) .$ Therefore, we obtain \begin{equation*}
	S^{(m)} N(S^{(m+n)}) \subseteq N(I_{E^{\ot m}} \ot  S^{(n)})  \cap S^{(m)}(\mathcal{H}).
	\end{equation*} On the other hand, let $ \xi \in  N(I_{E^{\ot m}} \ot  S^{(n)})  \cap S^{(m)}( \mathcal{H})$, then $ (I_{E^{\ot m}} \ot S^{(n)})\xi= 0 $ and $ \xi = S^{(m)}h$ for some $ h \in  \mathcal{H}.$ Consequently \begin{equation*}
	S^{(m+n)}h=(I_{E^{\ot m}} \ot S^{(n)})S^{(m)}h=
	(I_{E^{\ot m}} \ot S^{(n)})\xi= 0
	\end{equation*} which implies $  h \in N(S^{(m+n)})$. Hence $ \xi= S^{(m)} h \in S^{(m)}(N(S^{(m+n)})),$ so the opposite inclusion is verified.\qedhere
	
\end{proof}

The following result presents some valuable connections between the kernel of generalized inverse and the generalized range, which leads to defining the bi-regularity condition.

\begin{theorem}\label{S}
	Consider a CB-representation $(\sigma,V)$  $E$ on $\mathcal{H}$ and let $S$ be a generalized inverse of $\widetilde{V}$. Then the following statements are equivalent:
	\begin{enumerate}
		\item $ N(I_{E^{\ot m}} \ot S) \subseteq R(S^{(m)})$  for  each  $ m \in \mathbb{N}$ ;
		\item $ N(I_E \ot S^{(n)})\subseteq R(S)$  for  each  $ n \in \mathbb{N}$ ;
		\item $ N(I_{E^{\ot m}}  \ot S^{(n)}) \subseteq R(S^{(m)})$  for  each  $ n, m \in \mathbb{N}$ ;
		\item $ N(I_{E^{\ot m}}\ot S^{(n)})= S^{(m)}(N(S^{(m+n)}))$  for  each  $ n, m \in \mathbb{N}$.
	\end{enumerate}
\end{theorem}
\begin{proof}
	$ (1) \implies (2):$ We will prove it by Mathematical induction. For $ n= 1$, we need to show that \begin{equation*}
	N(I_E \ot S)\subseteq R(S),
	\end{equation*} which is true when we substitute $ m =1 $ in (1). Suppose that (2) is  valid for $ n=k$. We have to prove that (2) holds for $ n=k+1$, i.e., \begin{equation*}
	N(I_E \ot S^{(k+1)})\subseteq R(S) .
	\end{equation*}
	Let $ \xi \in N(I_E \ot S^{(k+1)}),$ then \begin{align*}
	(I_{E^{\ot k+1}} \ot S)(I_{E} \ot S^{(k)})(\xi)= (I_{E}\ot (I_{E^{\ot k}} \ot S)S^{(k)})(\xi)=(I_E \ot S^{(k+1)})(\xi)=0.
	\end{align*} Therefore $ (I_{E} \ot S^{(k)})(\xi) \in N(I_{E^{\ot k+1}} \ot S) \subseteq R(S^{(k+1)}),$ here last inequality follows from (1). Therefore, there is $ h \in N(S^{(k+2)})$ so that
	$(I_{E} \ot S^{(k)})(\xi)= S^{(k+1)}(h)=(I_{E} \ot S^{(k)})S(h).$
	Thus we get $ \xi-S(h) \in N(I_{E} \ot S^{(k)}) \subseteq R(S).$ Hence $ \xi \in R(S)$.
	
	$(2)\implies (3):$	We will prove inequality (3) by Mathematical induction.  For $ m=1 $, nothing to prove. Suppose (3) is valid for $ m =k.$ Now we need to show that it is valid for $m =k+1$, that is,\begin{equation*}
	N(I_{E^{\ot k+1}}  \ot S^{(n)}) \subseteq R(S^{(k+1)}).
	\end{equation*} Let $\xi\in N(I_{E^{\ot k+1}}  \ot S^{(n)}),$ then by (2)  $  	N(I_{E^{\ot k+1}}  \ot S^{(n)}) \subseteq (I_{E^{\ot k }} \ot S)(E^{\ot k+1} \ot \mathcal{H}) ,$ therefore there exists $ \eta\in N(I_{E^{\ot k}}\ot S^{(n+1)})  $ such that $\xi = (I_{E^{\ot k}}\ot S)(\eta)  $ and by induction hypothesis $ \eta\in R({S^{(k)}}). $ It follows that $ \xi = (I_{E^{\ot k}}\ot S)(\eta)\in R({S^{(k+1)}}) $ and hence $	N(I_{E^{\ot k+1}}  \ot S^{(n)}) \subseteq R(S^{(k+1)}).$

	$(3) \implies (4):$ Suppose $ N(I_{E^{\ot m}}  \ot S^{(n)}) \subseteq R(S^{(m)})$, for $ m, n \in \mathbb{N}.$ By Lemma \ref{kk} and (3), we get \begin{align*}
	S^{(m)} N(S^{(m+n)}) = N(I_{E^{\ot m}} \ot  S^{(n)})  \cap S^{(m)}( \mathcal{H})=  N(I_{E^{\ot m}} \ot  S^{(n)}),
	\end{align*} which proves (4).
	
	$ (4) \implies (1):$  Trivial.  
\end{proof}
The following definition draws inspiration from a recent article by Ezzahraoui, Mbekhta, and Zerouali in \cite{EMZ21}.
\begin{definition}
	Consider a regular CB-representation $(\sigma,V)$ of $E$ on $\mathcal{H}$ and let $S$ be a generalized inverse of $\widetilde{V}.$ We say that $(\sigma,V)$ is  {\rm bi-regular} if its generalized inverse $S$ satisfies any one of conditions of Theorem \ref{S}.
\end{definition}
\begin{theorem}\label{3.6}
	Consider a  regular CB-representation $(\sigma,V)$ of $E$ on $\mathcal{H}.$ If $S$ is the generalized inverse of $\widetilde{V},$
	then  $\widetilde{V}_{n}S^{(n)} \widetilde{V}_{n}= \widetilde{V}_{n},$ $ n \in \mathbb{N}.$
\end{theorem}
\begin{proof}
	
	For every $ k \geq 1$, we begin by demonstrating that  $(I_{E} \ot S^{(k)}) N(\widetilde{V}) \subseteq N(\widetilde{V}_{k+1})$.
	For that, we need to prove the following inequality \begin{equation}\label{(0.1)}
	(I_{E^{\otimes n}} \otimes S) N(\widetilde{V}_{n})\subseteq N(\widetilde{V}_{n+1}) ,\quad \mbox{where}\:\:\: n \in \mathbb{N}.
	\end{equation}
	Let
	$	\xi \in N(\widetilde{V}_{n})\subseteq(I_{E^{\ot n}} \ot \widetilde{V})(E^{\ot {n+1}} \ot \mathcal{H})$
	(using by Theorem \ref{cc}), there is $ \eta \in E^{\ot {n+1}} \ot \mathcal{H}$ such that $ \xi = (I_{E^{\ot n}} \ot \widetilde{V})(\eta) $. Observe that \begin{align*}\widetilde{V}_{n+1} (I_{E^{\otimes n}} \otimes S)(\xi) &= \widetilde{V}_{n}(I_{E^{\ot n}} \ot \widetilde{V}S\widetilde{V})(\eta)=\widetilde{V}_{n}(I_{E^{\ot n}} \ot \widetilde{V})( \eta) =0
	\end{align*}
	and hence  the Inequality (\ref{(0.1)}). 
	Using the  Inequality (\ref{(0.1)}), it is easy to see that
	\begin{equation}\label{(0.2)}
	(I_{E} \ot S^{(k)})N(\widetilde{V}) \subseteq N(\widetilde{V}_{k+1}).
	\end{equation}
	Now, we show that
	$\widetilde{V}_{n}S^{(n)}\widetilde{V}_{n}= \widetilde{V}_{n}$  for every  $ n \in \mathbb{N}.$ Consider \begin{align*}
	\widetilde{V}_{n} S^{(n)} \widetilde{V}_{n}  -  \widetilde{V}_{n} &=\widetilde{V}_{n}\big(\sum_{k=1}^{n} (I_{E^{\otimes k-1}} \ot S^{(n-k+1)} \widetilde{V}_{n-k+1})-(I_{E^{\otimes k}} \ot S^{(n-k)}  \widetilde{V}_{n-k})\big) \\&=\sum_{k=1}^{n} \widetilde{V}_{n} (I_{E^{\otimes k}} \ot S^{(n-k)})((I_{E^{\otimes k-1}}\ot S \widetilde{V})-I_{E^{\otimes k }\ot \mathcal{H}})(I_{E^{\otimes k}} \ot \widetilde{V}_{n-k}).
	\end{align*} Note that $ ((I_{E^{\otimes k-1}}\ot S \widetilde{V})-I_{E^{\otimes k }\ot \mathcal{H}})\xi_k \in N(\widetilde{V}_{k})$ for every $ \xi_k \in {E^{\otimes k }\ot \mathcal{H}}$ and Equations (\ref{(0.1)}) and (\ref{(0.2)}) follows that \begin{equation*}
	(I_{E^{\otimes k}} \ot S^{(n-k)}) ((I_{E^{\otimes k-1}}\ot S \widetilde{V})-I_{E^{\otimes k }\ot \mathcal{H}})\xi_k \in N(\widetilde{V}_{n}),
	\end{equation*} for $ 1 \leq k \leq n $. Therefore, \begin{equation*}
	\widetilde{V}_{n} (I_{E^{\otimes k}} \ot S^{(n-k)})((I_{E^{\otimes k-1}}\ot S \widetilde{V})-I_{E^{\otimes k }\ot \mathcal{H}})=0.
	\end{equation*} So, we get $\widetilde{V}_{n}S^{(n)}\widetilde{V}_{n}= \widetilde{V}_{n},$ for $ n \in \mathbb{N}.$
\end{proof}
\begin{corollary}
	Consider a bi-regular CB-representation $(\sigma,V)$ of $E$ on $\mathcal{H}$ and let $S$ be a generalized inverse of $\widetilde{V}.$ Then for each $ n\in \mathbb{N},$ $S^{(n)}$ is a generalized inverse of $\widetilde{V}_{n},$ that is,  \begin{align*}
	S^{(n)}\widetilde{V}_{n}S^{(n)}=S^{(n)} \quad	and \quad
	\widetilde{V}_{n}S^{(n)}\widetilde{V}_{n}= \widetilde{V}_{n}.
	\end{align*}
\end{corollary}
\begin{proof}
	Since $S$ is a generalized inverse of $\widetilde{V},$
	$
	\widetilde{V}S\widetilde{V}=\widetilde{V}$ and $ S\widetilde{V}S=S.
	$
	Now $(\sigma,V)$ is bi-regular therefore by Theorem \ref{3.6}, we can see that $
	\widetilde{V}_{n}S^{(n)}\widetilde{V}_{n}= \widetilde{V}_{n}$ for $  n \in \mathbb{N}.$ So it is sufficient to prove $
	S^{(n)}\widetilde{V}_{n}S^{(n)}=S^{(n)}
	$. For every $ k \in \mathbb{N}$, we begin by demonstrating that 
	
	\begin{align}\label{a2} \widetilde{V}_{k} (N(I_{E^{\ot k}} \ot S)) \subseteq N(S^{(k+1)}). \end{align}
	Let  $ \xi \in N(I_{E} \ot S^{(n)})\subseteq  R(S)$ for every $n \in \mathbb{N}$ (using the fact that $(\sigma,V)$ is bi-regular), there is $ h \in \mathcal{H}$ such that $ \xi = Sh $. Note that\begin{align*}S^{(n+1)} \widetilde{V}\xi = S^{(n+1)} \widetilde{V}Sh= (I_{E}\ot S^{(n)}) S\widetilde{V}Sh=(I_{E}\ot S^{(n)})Sh =(I_{E}\ot S^{(n)})\xi=  0.
	\end{align*}
	It implies
	\begin{equation}\label{a1}
	\widetilde{V}( N(I_{E} \otimes S^{(n)})) \subseteq N({S^{(n+1)}}) ,\quad  n \in \mathbb{N}.
	\end{equation}
	Using  Equation (\ref{a1}) and by Mathematical induction, we get
	(\ref{a2}).
	
	Now we show that
	$S^{(n)}\widetilde{V}_{n}S^{(n)} = S^{(n)}$ for $ n \in \mathbb{N}.$ Consider \begin{align*}
	S^{(n)}\widetilde{V}_{n}S^{(n)} - S^{(n)} &=S^{(n)}(\widetilde{V}_{n}S^{(n)}-I)=S^{(n)}\big(\sum_{k=1}^{n}\widetilde{V}_{n-k+1} S^{(n-k+1)}-\widetilde{V}_{n-k}S^{(n-k)}\big) \\&=S^{(n)}\sum_{k=1}^{n}(\widetilde{V}_{n-k}(I_{E^{\ot n-k} }\ot \widetilde{V})(I_{E^{\ot n-k}} \ot S) S^{(n-k)}-\widetilde{V}_{n-k}S^{(n-k)})\\&=\sum_{k=1}^{n}S^{(n)}\widetilde{V}_{n-k}((I_{E^{\ot n-k} }\ot \widetilde{V}S)-I_{E^{\ot n-k}\ot \mathcal{H}}) S^{(n-k)} .
	\end{align*} Note that $ ((I_{E^{\ot n-k} }\ot \widetilde{V}S)-I_{E^{\ot n-k}\ot \mathcal{H}})\xi_{n-k}=(I_{E^{\ot n-k} }\ot (\widetilde{V}S-I_{\mathcal{H}}))\xi_{n-k} \in N(I_{E^{\ot n-k}}\ot S)$ for every $ \xi_{n-k} \in {E^{\ot n-k}\ot \mathcal{H}}$, $ 1 \leq k \leq n $ and Equations (\ref{a1}) and (\ref{a2}) follows that \begin{equation*}
	\widetilde{V}_{n-k} ((I_{E^{\ot n-k} }\ot \widetilde{V}S)-I_{E^{\ot n-k}\ot \mathcal{H}})\xi_{n-k}\in N(S^{(n-k+1)})\subseteq N(S^{(n)}),
	\end{equation*} for every $ 1 \leq k \leq n $. Therefore, \begin{equation*}
	S^{(n)} \widetilde{V}_{n-k}((I_{E^{\ot n-k} }\ot \widetilde{V}S)-I_{E^{\ot n-k}\ot \mathcal{H}}) = 0.
	\end{equation*} So we deduce from the above fact that $S^{(n)}\widetilde{V}_{n}S^{(n)}= S^{(n)},$   $ n \in \mathbb{N}.$
\end{proof}
\begin{remark}
	\label{Inverse}
	Consider a regular CB-representation $(\sigma,V)$  of $E$ on $\mathcal{H}$ and let $S$ be a generalized inverse of $\widetilde{V}.$ Then \begin{align*}
	R^{\infty}({V})= \{h \in \mathcal{H}  :  \widetilde{V}_{n}S^{(n)}h =h, \quad for \quad all\quad n \in \mathbb{N}\}.
	\end{align*}
\end{remark}

Indeed, let $h \in R^{\infty}({V})$, then for each $n \in \mathbb{N}$ there is $ \xi_{n} \in {E^{\otimes n}} \ot \mathcal{H}$
such that $\widetilde{V}_{n} (\xi_{n})=h$, for every $n\in \mathbb{N} .$ Now, apply $ \widetilde{V}_{n}S^{(n)} $ in both sides, we get \begin{align*}&
\widetilde{V}_{n}S^{(n)}	\widetilde{V}_{n} \xi_{n}=\widetilde{V}_{n}S^{(n)}h.
\end{align*} Now by using Theorem \ref{3.6}, we have
$	\widetilde{V}_{n} \xi_{n}=\widetilde{V}_{n}S^{(n)}h.$ This implies $h=\widetilde{V}_{n}S^{(n)}h$, for all $ n\in \mathbb{N}.$

On the other hand, let $ h \in \mathcal{H}$ be such that $  \widetilde{V}_{n}S^{(n)}h =h $, for all $ n\in \mathbb{N}$. Then $ h \in R(\widetilde{V}_{n})$ and hence $h \in R^{\infty}({V}).$
Now we recall the definition of invariant subspace (cf. \cite{SZ08}) for the covariant representation $(\sigma, V).$
\begin{definition}	
	\begin{itemize}
		\item[$(1)$] Consider a CB-representation $(\sigma,V)$ of $E$ on a Hilbert space $\mathcal{H}$ and suppose  $\mathcal K$ is a closed subspace of $ \mathcal H.$  Then we say $\mathcal K$ is  $(\sigma, V)$-{\em invariant}   if it  is $\sigma(\mathcal B)$-invariant
		and, is invariant by each operator $V(\xi),\: \xi \in E.$  In addition, if  $ \mathcal{K}^{\bot}$ is   invariant by $V (\xi)$ for $\xi \in E,$ then we say $\mathcal{K}$ is $(\sigma,V)$-{\em reducing}.  Restricting naturally  this representation we get  another representation of $E$ on $\mathcal{K}$ which will be denoted as $(\sigma, V)|_{\mathcal{K}}.$ 
		
		\item[$(2)$]  A closed  subspace $\mathcal{W}$ of $\mathcal{H}$ is called {\em wandering} subspace  for $(\sigma , V)$, if it is $\sigma(\mathcal{B})$-invariant and $\mathcal{W}\perp\wt{V}_{n}(E^{\ot n} \ot \mathcal{W})$ for every $n \in \mathbb{N}.$ The representation $(\sigma, V)$ has {\em generating wandering subspace property} (GWS-property) if there is a wandering subspace $\mathcal{W}$ of $\mathcal{H}$ satisfying
		$$\mathcal{H}=\displaystyle\bigvee_{n\in \mathbb{N}_0}\wt{V}_{n}(E^{\ot n} \ot \mathcal{W})$$
		and the corresponding  wandering subspace $\mathcal{W}$
		is called generating wandering subspace (GWS). 
	\end{itemize}
	
\end{definition}
\begin{corollary}\label{invariant}
	Consider a regular CB-representation $(\sigma,V)$  of $E$ on $\mathcal{H}$ and let $S$ be a generalized inverse of $\widetilde{V}$. Then \begin{align*}
	S (R^{\infty}({{V}})) \subseteq E\ot R^{\infty}({{V}}).
	\end{align*}
\end{corollary}
\begin{proof}
	Let $h \in R^{\infty}({{V}}),$ then by Remark \ref{Inverse}, $\widetilde{V}Sh=h $ and thus there exist  $\xi_n \in E^{\ot n }  \ot \mathcal{H}, n \in \mathbb{N}, $ such that \begin{align*}
	\widetilde{V}Sh = \widetilde{V}_{n+1}(\xi_{n+1})=\widetilde{V}(I_{E} \ot \widetilde{V}_{n})(\xi_{n+1}) ,
	\end{align*}  which implies that \begin{align*}
	Sh-(I_{E} \ot \widetilde{V}_{n})(\xi_{n+1}) \in N(\widetilde{V}) \subseteq (I_{E} \otimes \widetilde{V}_n)(E^{\otimes(n+1) }\otimes \mathcal{H}),
	\end{align*} here the last inequality follows from Theorem \ref{cc}. Let $\eta:= Sh-(I_{E} \ot \widetilde{V}_{n})(\xi_{n+1}).$ Therefore, \begin{align*}
	Sh=	\eta +(I_{E} \ot \widetilde{V}_{n})(\xi_{n+1}) \in (I_{E} \otimes \widetilde{V}_n)(E^{\otimes(n+1) }\otimes \mathcal{H}), \quad  n \in \mathbb{N}.
	\end{align*}  Hence $ Sh \in E \ot R(\widetilde{V}_{n}) $  for every $ n \in \mathbb{N}$,  $ Sh \in E \ot  R^{\infty}({{V}}) $. Thus $
	S (R^{\infty}({{V}})) \subseteq E\ot R^{\infty}({{V}}).
	$ \end{proof}
\section{ The Moore-Penrose inverse and the reduced minimum modulus}\label{section 5}
We begin this section by defining the Moore-Penrose inverse for $\widetilde{V}$ and brief details of the Moore-Penrose inverse (for more details see \cite{D19}).

Consider a CB-representation $(\sigma,V)$  of $E$ on $\mathcal{H}$  such that $\wV$ has closed range. If there is a unique operator $\widetilde{V}^\dagger \in \mathcal{B}(\mathcal{H},E \ot \mathcal{H})$ so that \begin{enumerate}
	\item $N({\widetilde{V}^\dagger}) = R({\widetilde{V}})^\perp=N({\widetilde{V}^*})$ and
	\item $\widetilde{V}^\dagger \widetilde{V}\xi=\xi$, \quad for $\xi \in N({\widetilde{V}})^\perp,$
\end{enumerate}
then the operator $ \widetilde{V}^\dagger$ is called the {\em Moore-Penrose inverse} (MPI) of $\widetilde{V}$ and it is satisfies the following: \begin{equation*}
\widetilde{V}=	\widetilde{V}\widetilde{V}^{\dagger}\widetilde{V},\quad \widetilde{V}^{\dagger}=\widetilde{V}^{\dagger}\widetilde{V}\widetilde{V}^{\dagger},\quad \widetilde{V}\widetilde{V}^{\dagger}=(\widetilde{V}\widetilde{V}^{\dagger})^*,\quad \widetilde{V}^{\dagger}\widetilde{V}=( \widetilde{V}^{\dagger}\widetilde{V})^*.
\end{equation*}

\begin{remark}
	Observe that 
	$\widetilde{V}^\dagger =  ({\widetilde{V}^*\widetilde{V}})^{-1}\widetilde{V}^*$ when 	$\widetilde{V}$ has left inverse and if $\widetilde{V}$ has right inverse, then	$\widetilde{V}^\dagger$ will be right inverse and equal to  $\widetilde{V}^*({\widetilde{V}\widetilde{V}^*})^{-1}.$
\end{remark}
In the following definition, we introduce a notion of reduced minimum modulus (cf. \cite{A07,A85}) for the CB-representations of $E$ on $\mathcal{H}$:

\begin{definition}
	Consider a CB-representation $(\sigma,V)$ of $E$ on $\mathcal{H}.$ The {\em  reduced minimum modulus} $\gamma({\widetilde{V}})$ of  $(\sigma,V)$ is defined by
	$$	\gamma({\widetilde{V}}):=\begin{cases}
	\inf_{\xi \notin N({\widetilde{V}})} \frac{\|\widetilde{V}\xi\|}{\dis (\xi , N({\widetilde{V}}))}  & \text{if }   \wV\neq 0 \\
	\infty  & \text{if }  \widetilde{V} = 0.
	\end{cases}$$
	
\end{definition}
It is easy to see that $\widetilde{V}$ has  closed range if and only if $\gamma({\widetilde{V}}) \textgreater 0.$

\begin{proposition}\label{5.3}
	Consider a regular CB-representation $(\sigma,V)$  of $E$ on $\mathcal{H}.$ Then for each $n \in \mathbb{N}$, \begin{equation*}
	\gamma({\widetilde{V}}_n) \geq \gamma({\widetilde{V}}) \gamma(I_{E} \ot {\widetilde{V}})  \cdots \gamma(I_{E^{\ot n-1}} \ot {\widetilde{V}}).
	\end{equation*}
\end{proposition}
\begin{proof}
	We prove it by  Mathematical induction. The case $ n =1 $ is trivial. Assume that \begin{equation*}
	\gamma({\widetilde{V}}_n) \geq \gamma({\widetilde{V}}) \gamma(I_{E} \ot {\widetilde{V}}) \cdots \gamma(I_{E^{\ot n-1}} \ot {\widetilde{V}}).
	\end{equation*} For $ \xi \in E^{\ot(n+1)} \ot \mathcal{H} $ and $ \eta \in N({\widetilde{V}}_{n+1}) ,$ we have \begin{align*}
	\dis( \xi , N({\widetilde{V}}_{n+1})) &= \dis( \xi- \eta , N({\widetilde{V}}_{n+1})) \leq 	\dis( \xi - \eta , N(I_{E^{\ot n}} \ot {\widetilde{V}})),
	\end{align*} as $ N(I_{E^{\ot n}} \ot {\widetilde{V}})\subseteq N({\widetilde{V}}_{n+1}), $ for every $ n \in \mathbb{N} $. By assumption $(\sigma,V)$ is regular and  by (4) of Theorem \ref{cc},  \begin{equation*}
	N({\widetilde{V}_{n}}) = (I_{E^{\ot n}} \ot {\widetilde{V}})(N(\widetilde{V}_{n+1}))
	\end{equation*} and therefore \begin{align*}
	\dis((I_{E^{\ot n}} \ot {\widetilde{V}})(\xi) , N({\widetilde{V}_{n}})) &= 	\dis((I_{E^{\ot n}} \ot {\widetilde{V}})(\xi) , (I_{E^{\ot n}} \ot {\widetilde{V}})(N(\widetilde{V}_{n+1}))) \\&= \inf_{ \eta \in N({\widetilde{V}}_{n+1})}\|(I_{E^{\ot n}} \ot {\widetilde{V}})(\xi) -(I_{E^{\ot n}} \ot {\widetilde{V}})(\eta)\| \\&= \inf_{  \eta \in N({\widetilde{V}}_{n+1})}\|(I_{E^{\ot n}} \ot {\widetilde{V}})(\xi -\eta)\| \\& \geq \gamma(I_{E^{\ot n}} \ot {\widetilde{V}}) \inf_{ \eta \in N({\widetilde{V}}_{n+1})} \dis((\xi -\eta) ,N(I_{E^{\ot n}} \ot {\widetilde{V}}))\\& \geq \gamma(I_{E^{\ot n}} \ot {\widetilde{V}})  \dis(\xi ,  N({\widetilde{V}}_{n+1})).
	\end{align*} From the above observation, we get \begin{align*}
	\|\widetilde{V}_{n+1}(\xi)\| =\|\widetilde{V}_{n}(I_{E^{\ot n}} \ot {\widetilde{V}})( \xi)\| &\geq \gamma({\widetilde{V}_{n}}) \dis((I_{E^{\ot n}} \ot {\widetilde{V}})( \xi) , N({\widetilde{V}_{n}}))\\& \geq \gamma({\widetilde{V}_{n}}) \gamma(I_{E^{\ot n}} \ot {\widetilde{V}}) \dis(\xi ,  N({\widetilde{V}}_{n+1})).
	\end{align*} Consequently, from our induction hypothesis \begin{equation*}
	\gamma({\widetilde{V}}_{n+1})\geq   \gamma({\widetilde{V}_{n}}) \gamma(I_{E^{\ot n}} \ot {\widetilde{V}})\geq  \gamma({\widetilde{V}}) \gamma(I_{E} \ot {\widetilde{V}}) \gamma(I_{E^{\ot 2}} \ot {\widetilde{V}}) \cdots \gamma(I_{E^{\ot n}} \ot {\widetilde{V}}),
	\end{equation*} which completes the proof.
\end{proof}
\begin{corollary}
	Consider a regular CB-representation $(\sigma,V)$  of $E$ on $\mathcal{H}.$ Then \begin{enumerate}
		\item $ R^{\infty}({V}) $ is closed.
		\item   $\widetilde{V}$ is left invertible if  $ R^{\infty}({V})= 0 .$
	\end{enumerate}
\end{corollary}
\begin{proof}
	(1)	Since  $(\sigma,V)$ is a regular CB-representation of $E$ on $\mathcal{H}$, $R(\wV)$ is closed. Thus for each $ n\in \mathbb{N},$ we have $R(I_{E^{\ot n}}\ot \wV)$ is closed and hence  using Proposition \ref{5.3}, $\gamma({{\widetilde{V}}_n})  \textgreater 0 $. From this, we conclude that  $ R({\widetilde{V}}_{n})$ is closed for every $ n \in \mathbb{N} ,$ and therefore $ R^{\infty}({V}) $ is closed.\\(2)  Since  $ R^{\infty}({V})= 0 $ and $(\sigma,V)$ is  regular, $\widetilde{V}$ is injective and  $R(\widetilde{V})$ is closed. Therefore $\widetilde{V}$ is left invertible.
\end{proof}
The following proposition establishes the relationship between the reduced minimum modulus and MPI (see \cite[Corollary 2.3]{M13}). The proof of the following proposition is similar to the operator theory case, so we can omit it.
\begin{proposition}\label{Regular}
	Consider a  CB-representation $(\sigma,V)$  of $E$ on $\mathcal{H}$  such that $\wV$ has closed range. Then \begin{equation*}
	\|\widetilde{V}^{\dagger}\|=\frac{1}{\gamma(\widetilde{V})}.
	\end{equation*}
\end{proposition}  The following proposition explains numerous essential facts about MPI of $\widetilde{V},$ which will come in handy throughout the sequel. We omit the proof of this proposition, which is straightforward.

\begin{proposition}
	Suppose $(\sigma,V)$ is a CB-representation of $E$ on $\mathcal{H}$  such that $\wV$ has closed range. Then
	\begin{enumerate}
		\item  $\widetilde{V}\widetilde{V}^{\dagger}=P_{R({\widetilde{V}})}$ and  $\widetilde{V}^{\dagger}\widetilde{V}=P_{N({\widetilde{V})}^{\perp}}$,
		\item $ R({\widetilde{V}}^\dagger) =R({\widetilde{V}}^*)=N({\widetilde{V})}^{\perp}$,
		\item$ N(\widetilde{V}^{\dagger})=N(\widetilde{V}\widetilde{V}^{\dagger})=N({\widetilde{V}}^*)=R({\widetilde{V}})^{\perp}$,
		\item$R(\widetilde{V})=R(\widetilde{V}\widetilde{V}^{\dagger})= R({\widetilde{V}}^{\dagger*})$,
		\item$N({\widetilde{V}})=N(\widetilde{V}^\dagger\widetilde{V})=N({\widetilde{V}}^{\dagger*})$,
		\item$({\widetilde{V}}^*)^\dagger=({\widetilde{V}}^\dagger)^*$,
		\item$({\widetilde{V}}^\dagger)^\dagger=\widetilde{V}$,
		\item$\widetilde{V}^*\widetilde{V}\widetilde{V}^{\dagger}=\widetilde{V}^\dagger\widetilde{V}\widetilde{V}^*=\widetilde{V}^*.$
	\end{enumerate}
\end{proposition}
Consider a  CB-representation $(\sigma,V)$  of $E$ on $\mathcal{H}$ with $\widetilde{V}$ has closed range. For $n \in \mathbb{N},$ define
\begin{align*}
\widetilde{V}^{\dagger(n)}:= (I_{E^{\otimes n-1}}\ot \widetilde{V}^{\dagger})(I_{E^{\otimes n-2}}\ot\widetilde{V}^{\dagger}) \cdots(I_{E}\ot \widetilde{V}^{\dagger})\widetilde{V}^{\dagger}.
\end{align*}
Note that  the equality $\widetilde{V}^{\dagger(n)}=\widetilde{V}_n^{\dagger}, n \geq 2,$ is not true even if  $\wt{V}$ is left invertible (cf. \cite[Example 4]{EMZ21}), where $\widetilde{V}_n^{\dagger}$ is MPI of $\wt{V}_n.$ Moreover, if $\wt{V}$ is left invertible, so is $\wt{V}_n.$  Thus  $$\widetilde{V}^{\dagger(n)}=(I_{E^{\otimes n-1}}\ot (\wt{V}^*\wt{V})^{-1}\wt{V}^*)(I_{E^{\otimes n-2}}\ot(\wt{V}^*\wt{V})^{-1}\wt{V}^*) \cdots(I_{E}\ot (\wt{V}^*\wt{V})^{-1}\wt{V}^*)(\wt{V}^*\wt{V})^{-1}\wt{V}^*,$$
but $\wt{V}^{\dagger}_n=(\wt{V}_n^*\wt{V}_n)^{-1}\wt{V}_n^*.$ For this purpose, we define the following:
For $n \in \mathbb{N},$ the covariant representation $(\sigma, V)$ is said  to be  {\em $n$-dagger} if $\widetilde{V}^{\dagger(n)}=\widetilde{V}_n^{\dagger},$ and is said to be  {\em hyper-dagger} if it is  $n$-dagger for all $n \in \mathbb{N}.$

\begin{proposition}\label{invert}
	Consider a regular CB-representation $(\sigma,V)$  of $E$ on $\mathcal{H}$. If $R({\widetilde{V}^{\dagger({n})}}) \subseteq N({\widetilde{V}_{n}})^{\perp}, n \geq 2 ,$ then the restriction map  \begin{align*}
	{\widetilde{V}}_n|_{({{E^{\otimes{n}} \otimes R^{\infty}{({V})})\cap R({\widetilde{V}^{\dagger({n})})}}}}:{({{E^{\otimes{n}} \otimes R^{\infty}{({V})})\cap R({\widetilde{V}^{\dagger({n})})}}}}\rightarrow R^\infty{({V})}
	\end{align*} is bijective.  In particular, \begin{align*}
	\widetilde{V}|_{({{E \otimes R^{\infty}{({V})})}\cap N({\widetilde{V}})}^\perp}:{({{E \otimes R^{\infty}{({V})})}\cap N({\widetilde{V}})}^\perp}\rightarrow R^\infty{({V})}\end{align*} is bijective.
\end{proposition}
\begin{proof}
	
	Since  $(\sigma,V)$ is regular,  by Theorem \ref{reg},
	$ \widetilde{V}({E}\ot { {{R}^{\infty}{({V})}}}) = { {{R}^{\infty}{({V})}}}$ and thus $ \widetilde{V}_{n}({E^{\ot{n}}}\ot { {{R}^{\infty}{({V})}}}) = { {{R}^{\infty}{(\widetilde{V})}}},$ for $n \in \mathbb{N}.$ Also, \begin{align*}
	\widetilde{V}_{n}(({E^{\otimes{n}} \otimes R^{\infty}{({V})})\cap R({\widetilde{V}^{\dagger({n})}})}) \subseteq \widetilde{V}_{n}({E^{\otimes{n}} \otimes R^{\infty}{({V})}})=R^{\infty}{({V})}.
	\end{align*}

	For $ h \in R^{\infty}{({V})}$ and using Theorem  $\ref{Inverse}$, $ h = \widetilde{V}_{n}\widetilde{V}^{\dagger (n)} h, n \in \mathbb{N}.$ By Corollary $\ref{invariant}$, $\widetilde{V}^{\dagger}(R^{\infty}(\wV)) \subseteq E \ot R^{\infty}(\wV)$ and hence $\widetilde{V}^{\dagger (n)} (h) \in (E^{\otimes{n}}\ot R^{\infty}({{V}})) \cap R({\widetilde{V}^{\dagger({n})})}.$ Therefore  \begin{equation*}
	h = \widetilde{V}_{n}(\widetilde{V}^{\dagger (n)} h )\in \widetilde{V}_{n} ((E^{\otimes{n}}\ot R^{\infty}({{V}})) \cap R({\widetilde{V}^{\dagger({n})})}),
	\end{equation*}  which proves ${\widetilde{V}}_n|_{({E^{\otimes{n}} \otimes R^{\infty}{({V})})}\cap R({\widetilde{V}^{\dagger({n})})}}$ is onto. Now, we need to show that $	{\widetilde{V}}_n|_{({E^{\otimes{n}} \otimes R^{\infty}{({V})})}\cap R({\widetilde{V}^{\dagger({n})})}}$ is injective. Suppose $ \xi \in ({E^{\otimes{n}} \otimes R^{\infty}{({V})})\cap R({\widetilde{V}^{\dagger({n})})}}$ is such that $	{\widetilde{V}}_n{|} _{({{E^{\otimes{n}} \otimes R^{\infty}{({V})}})}\cap R({\widetilde{V}^{\dagger({n})})}}\xi=0.$ Then  \begin{equation*}
	\xi \in ({E^{\otimes{n}} \otimes R^{\infty}{({V})})\cap R({\widetilde{V}^{\dagger({n})})}}\cap N(\widetilde{V}_{n}).
	\end{equation*} By hypothesis $R({\widetilde{V}^{\dagger({n})}}) \subseteq N({\widetilde{V}_{n}})^{\perp},$ \begin{align*}
	\xi \in ({E^{\otimes{n}} \otimes R^{\infty}{({V})})\cap R({\widetilde{V}^{\dagger({n})})}}\cap N(\widetilde{V}_{n})= R({\widetilde{V}^{\dagger({n})})} \cap N(\widetilde{V}_{n}) = \{0\}.
	\end{align*}  
	This shows that  $	{\widetilde{V}}_n{|} _{({E^{\otimes{n}} \otimes R^{\infty}{({V})})}\cap R({\widetilde{V}^{\dagger({n})})}}$ is injective.
	Note that, for $n=1,$ $R({\widetilde{V}^{\dagger}}) = N({\widetilde{V}})^{\perp}.$ Hence \\$	 {\widetilde{V}}_n{|} _{({E^{\otimes{n}} \otimes R^{\infty}{({V})})}\cap R({\widetilde{V}^{\dagger({n})})}}$ is bijective, for $n \in \mathbb{N}.$
\end{proof}

\begin{proposition}\label{dagger}
	Suppose $(\sigma,V)$ is a regular CB-representation of $E$ on $\mathcal{H}.$ Then \begin{align*}
	(\widetilde{V}|_{ E \ot R^{\infty}({V})})^\dagger=\widetilde{V}^{\dagger}|_ {R^{\infty}({V})}.
	\end{align*}
\end{proposition}
\begin{proof}
	Since $(\sigma,V)$ is a regular CB-representation of $E$ on $\mathcal{H}$, $ N (\widetilde{V}|_{E \ot R^{\infty}({V})})= N (\widetilde{V}). $ Let $ A = \widetilde{V}|_{E \ot R^{\infty}({V})}$, then $ N(A) = N(\widetilde{V}) $ and by Theorem \ref{reg}, $ R (A) = R^{\infty}({V})$. An operator $A_0=A|_{{N(A)}^\perp} : {N(A)}^\perp \rightarrow R(A) $ is bijection and MPI of $ A ,$  $ A^\dagger ={ A_0}^{-1} P_{R^{\infty}({V})}.$ Since $ \widetilde{V}^\dagger ={ \widetilde{V}_0}^{-1} P_{R(\widetilde{V})} $ where $ \widetilde{V}_{0}=\widetilde{V}|_{N ({\widetilde{V}})^\perp} : N ({\widetilde{V}})^\perp \rightarrow R(\widetilde{V}) $, thus $ \widetilde{V}^\dagger |_{ R^{\infty}({V}) } ={ \widetilde{V}_0}^{-1} P_{R^{\infty}({V})} $.

	Let $ h \in R^{\infty}({V})$ then $ \widetilde{V}^\dagger | _{R^{\infty}({V})}h = {\widetilde{V}_0}^{-1} h $ and there exists $ \xi \in E \ot  R^{\infty}({V}) \cap N (\widetilde{V})^{\perp} $ such that $ \widetilde{V}_0^{{-1}} h = \xi $. Now, consider $ \eta = A^{\dagger}h = A_0^{-1} h $ then $ \widetilde{V}_0 \xi= A_0 \eta $, since $ \widetilde{V}_0 | _{ E \ot R^{\infty}({V}) }=A_{0 } $ and $\widetilde{V}_0 $ is bijective,  we deduce that $\xi=\eta$. Therefore $\widetilde{V}^{\dagger}| _{R^{\infty}({V})}=A^{\dagger}, $   $	(\widetilde{V}|_{ E \ot R^{\infty}({V})})^\dagger=\widetilde{V}^{\dagger}|_{ R^{\infty}({V})}$.
\end{proof}
\begin{theorem}\label{Theorem 1.7}
	Consider a  CB-representation $(\sigma,V)$  of $E$ on $\mathcal{H}$ such that $\wV$ has closed range. The following are equivalent: \begin{enumerate}
		\item $(\sigma ,V )$ is regular.
		\item  	The map $\hat{\widetilde{V}_{n}} :E^{\ot n}\ot R(\widetilde{V})^\perp \rightarrow R({\widetilde{V}}_{n})\cap R({\widetilde{V}}_{n+1})^\perp	$ defined by
		\begin{align*}
		\hat{\widetilde{V}_{n}}{\xi}= P_{ R({\widetilde{V}}_{n})\cap R({\widetilde{V}}_{n+1})^\perp} \widetilde{V}_{n}{\xi}, \:\:\:\:\xi \in E^{\ot n}\ot R(\widetilde{V})^\perp,\:\: n \in \mathbb{N}
		\end{align*}  is an invertible map.
	\end{enumerate}
\end{theorem}
\begin{proof}
	(1) $\implies $ (2): Since $(\sigma,V)$ is regular,  $ R({\widetilde{V}}_{n}) $ is closed and hence $ R({\widetilde{V}}_{n})\cap R({\widetilde{V}}_{n+1})^\perp$ is closed. Therefore $P_{ R({\widetilde{V}}_{n})\cap R({\widetilde{V}}_{n+1})^\perp}$ is continuous, so we conclude that $ \hat{\widetilde{V}_{n}} $ is well-defined bounded operator for every $n\in \mathbb{N}$.

	First we will show that $\hat{\widetilde{V}_{n}}$ is injective. Let $ \xi \in E^{\ot n}\ot R(\widetilde{V})^\perp $ be such that  $\hat{\widetilde{V}_{n}}{(\xi)}=0 ,$ i.e., $ P_{ R({\widetilde{V}}_{n})\cap R({\widetilde{V}}_{n+1})^\perp} {\widetilde{V}}_{n}{(\xi)}=0 .$ Since $ R({\widetilde{V}}_{n})=R({\widetilde{V}}_{n+1}) \oplus (R({\widetilde{V}}_{n})\cap R({\widetilde{V}}_{n+1})^\perp) $, $ {\widetilde{V}}_{n}{(\xi)} \in R({\widetilde{V}}_{n+1})$ and  there exists $ \eta \in E^{\ot {n+1}}\ot \mathcal{H} $ such that ${\widetilde{V}}_{n} (\xi)=\widetilde{V}_{n+1} (\eta) $. Using $(\sigma,V)$ is  regular, $ \xi - (I_{E^{\ot n}} \ot \widetilde{V}) \eta \in E^{\ot n} \ot R(\widetilde{V})$ but   $ \xi \in E^{\ot n}\ot R(\widetilde{V})^\perp, $ thus  $ \xi=0 $ and hence $\hat{\widetilde{V}_{n}} $ is injective.

	Now we need to show that $\hat{\widetilde{V}_{n}} $ is surjective. For $ h \in R({\widetilde{V}}_{n})\cap R({\widetilde{V}}_{n+1})^\perp $ is non-zero, $ i.e., h=P_{R({\widetilde{V}}_{n})\cap R({\widetilde{V}}_{n+1})^\perp}h$. To show that $ {\widetilde{V}}_{n}^\dagger h\notin E^{\ot n}\ot R(\widetilde{V})$. Suppose that  ${\widetilde{V}}_{n}^\dagger h = (I_{E^{\ot n}}\ot \widetilde{V})(\xi) $ for some $ \xi \in E^{\ot (n+1)} \ot \mathcal{H} $. 
	Since $ {\widetilde{V}}_{n}{\widetilde{V}}_{n}^\dagger $ is the orthogonal projection on $R({\widetilde{V}}_n) $,  $ h =  {\widetilde{V}}_{n}{\widetilde{V}}_{n}^\dagger h = {\widetilde{V}}_{n}(I_{E^{\ot n}}\ot \widetilde{V})(\xi)= \widetilde{V}_{n+1}(\xi) \in R({\widetilde{V}}_{n+1})$. Thus $ h \in R({\widetilde{V}}_{n+1})^\perp \cap R({\widetilde{V}}_{n+1})$  and hence $ h= 0$, which is  absurd due to $h \neq 0.$  Therefore ${\widetilde{V}}_{n}^\dagger h \notin E^{\ot n}\ot R(\widetilde{V})$ and ${\widetilde{V}}_{n}^\dagger h =  \eta +  (I_{E^{\ot n}}\ot \widetilde{V})(\xi)  $ for some $ \xi \in E^{\ot (n+1)} \ot \mathcal{H}, \eta \in E^{\ot n}\ot R(\widetilde{V})^\perp $. Now $ h={\widetilde{V}}_{n}{\widetilde{V}}_{n}^\dagger h  ={\widetilde{V}}_{n}(\eta) + \widetilde{V}_{n+1}(\xi) $ and $ h=P_{R({\widetilde{V}}_{n})\cap R({\widetilde{V}}_{n+1})^\perp}{\widetilde{V}}_n(\eta),$ we conclude that $h=\hat{\widetilde{V}_{n}} (\eta) $ for some $ \eta \in E^{\ot n}\ot R(\widetilde{V})^\perp .$ Therefore  $\hat{\widetilde{V}_{n}}$ is onto.

	(2) $\implies (1)$: Suppose $\hat{\widetilde{V}_{n}} :E^{\ot n}\ot R(\widetilde{V})^\perp \rightarrow R({\widetilde{V}}_{n})\cap R({\widetilde{V}}_{n+1})^\perp	$ is an invertible operator. For $ \xi\in N({\widetilde{V}}_{n}),$ then we have to show that $ \xi \in E^{\ot n}\ot R({\widetilde{V}}).$ Since $ R({\widetilde{V}}) $ is closed,  $ \xi= \eta +  (I_{E^{\ot n}}\ot \widetilde{V})(\kappa)  $ for some $ \kappa \in E^{\ot (n+1)} \ot \mathcal{H} $  and $ \eta \in E^{\ot n}\ot R(\widetilde{V})^\perp $. Then ${\widetilde{V}}_{n} (\eta)+\widetilde{V}_{n+1}(\kappa) =  {\widetilde{V}}_{n}(\eta +  (I_{E^{\ot n}}\ot \widetilde{V})(\kappa)) =  \widetilde{V}_n \xi =0 $ and hence $ 0 = P_{R({\widetilde{V}}_{n})\cap R({\widetilde{V}}_{n+1})^\perp}\widetilde{V}_{n+1}(\kappa) = -P_{R({\widetilde{V}}_{n})\cap R({\widetilde{V}}_{n+1})^\perp}{\widetilde{V}}_{n}(\eta) $, therefore $ \eta \in N (\hat{\widetilde{V}_{n}}).$ But by hypothesis, $\hat{\widetilde{V}_{n}}$ is injective,  $ \eta = 0 .$ Thus $ \xi=(I_{E^{\ot n}}\ot \widetilde{V})(\kappa) \in E^{\ot n}\ot R({\widetilde{V}}),$ which deduce that  $(\sigma , V)$ is regular. \end{proof}
\begin{corollary}
	Consider a regular CB-representation $(\sigma,V)$  of $E$ on $\mathcal{H}.$ Then \begin{equation*}
	\dim R{({\widetilde{V}}_{n})\cap R({\widetilde{V}}_{n+1})^\perp} = \dim({E^{\ot n} \ot R({\widetilde{V}})^\perp}), \:\:\: n\in \mathbb{N}.
	\end{equation*}
\end{corollary}
\section{Shimorin-Wold-type decomposition for regular CB-representation of $E$ on $\mathcal{H}$}\label{section 6}

Suppose $(\sigma,V)$ is a CB-representation of $E$ on a Hilbert space $\mathcal{H}$ such that $\wt{V}$ has  closed range. We use the notations $\mathcal{W}^{\dagger}$ and $\mathcal{W}$  throughout this section for   \begin{equation*}
\mathcal{W}^{\dagger}: = (E\ot \mathcal{H}) \ominus \widetilde{V}^\dagger({ \mathcal{H}}) = R({\widetilde{V}^\dagger})^\perp = N({\widetilde{V}}^{\dagger*})=N({\widetilde{V}}),
\end{equation*}  \begin{equation*}
\mathcal{W}: =\mathcal{H} \ominus \widetilde{V}({E \ot \mathcal{H}}) = R({\widetilde{V}})^\perp = N({\widetilde{V}}^*).
\end{equation*} 

Suppose that $\gamma(\widetilde{V})\geq 1,$ then by Proposition \ref{Regular},
$\widetilde{V}^\dagger $ is contraction and hence  \begin{equation}\label{cont}
\|\widetilde{V}^\dagger\widetilde{V}\xi\|^2\leq \|\widetilde{V}^\dagger\|^2 \|\widetilde{V}\xi\|^2\leq\|\widetilde{V}\xi\|^2,\:\:\: \xi\in E\ot \mathcal{H}.
\end{equation} 
As $\wt{V}^{\dagger}\wt{V}$ is the orthogonal projection onto $ N({\widetilde{V}})^{\perp} ,$  $ \widetilde{V}^* \widetilde{V}- \widetilde{V}^\dagger \widetilde{V}\geq 0 .$ We denote  $ D_{\widetilde{V}}$ by $ (\widetilde{V}^* \widetilde{V}- \widetilde{V}^\dagger \widetilde{V})^{\frac{1}{2}},$  i.e. $D_{\widetilde{V}}:= (\widetilde{V}^* \widetilde{V}- \widetilde{V}^\dagger \widetilde{V})^{\frac{1}{2}}.$ Clearly for every $\xi \in   E\ot \mathcal{H},$ \begin{equation*}
\|D_{\widetilde{V}}\xi\|^2= \|\widetilde{V}\xi\|^2- \|\widetilde{V}^\dagger\widetilde{V}\xi\|^2.
\end{equation*}
Consider a CB-representation $(\sigma,V)$  of $E$ on  $\mathcal{H}.$ We call it  is {\em concave} (cf. \cite{HV19}) if the following inequality \begin{equation*}
\|\widetilde{V}_{2}(\xi)\|^{2}+\|\xi\|^{2}\leq 2\|(I_{E}\ot \widetilde{V})\xi\|
\end{equation*} holds for all $\xi \in E^{\ot 2}\ot \mathcal{H}.$
If $(\sigma,V)$ is  isometric  then it is concave.

The following lemma in the article  \cite[Lemma 2.2]{HV19} provides a crucial aspect of the concave CB-representation of $E$ on $\mathcal{H}.$
\begin{lemma}\label{L1}
	Consider a concave CB-representation $(\sigma, V)$   of $E$ on  $\mathcal{H}.$ Then
	\begin{align}\label{CON}
	\|\wt{V}_k(\xi_k)\|^2 \leq \| \xi_k\|^2 + k(\|(I_{E^{\otimes (k-1)}} \otimes \wt{V})(\xi_k)\|^2-\|\xi_k\|^2), ~~\xi_{k} \in E^{\otimes k}\ot \mathcal{H}, k \in \mathbb{N}. \end{align}
\end{lemma}
Note that $\sum_{n=1}^{\infty}\frac{1}{n}= \infty.$ Now we can take general sequence, $(d_k)_k ,k\geq 2, $ of non-negative terms such that $\sum_{m \geq 2}$ $\frac{1}{d_m} = \infty,$ in place of the sequence $\frac{1}{n},$ which generalize  the Inequality (\ref{CON}) in more general form. At the beginning of this section, we discuss the results related to this generalization of Inequality (\ref{CON}).

The following proposition is a abstraction of \cite[Lemma 1]{R88}, \cite[Proposition 2.1]{O05}, \cite[Proposition 7]{EMZ15} and \cite[Lemma 2.2]{HV19}:
\begin{proposition}
	Consider a CB-representation $(\sigma,V)$  of $E$ on $\mathcal{H}$ so that $\gamma({\widetilde{V}}) \geq 1 .$ Suppose $ d \textgreater 0 $ and $(d_k)_k ,(k\geq 2) $ is some non-negative sequence. Then the following two inequalities are equivalent: \begin{equation}\label{eqn 2.1}
	\|\widetilde{V}_k (\xi_{k})\|^2 \leq d_k \| (I_{E^{\ot {k-1}}} \ot D_{\widetilde{V}}) (\xi_{k})\|^2 +d \|(I_{E^{\ot {k-1}}} \ot \widetilde{V}^\dagger \widetilde{V}) (\xi_{k})\|^2 
	\end{equation} for every $ \xi_{k} \in  E^{\ot k} \ot \mathcal{H},$  and \begin{equation}\label{2.2}
	\|\widetilde{V}_k (\eta_{k})\|^2 \leq d_k (\| (I_{E^{\ot {k-1}}} \ot \widetilde{V}) \eta_{k}\|^2 - \| \eta_{k}\|^2 )+ d \| \eta_{k}\|^2 ,
	\end{equation} for every $ \eta_{k} \in  E^{\ot{k-1}} \ot N({\widetilde{V}})^\perp .$
\end{proposition}
\begin{proof}
	Due to the fact that $\gamma(\widetilde{V})$ is strictly positive, the operator $\widetilde{V}^\dagger$ exists. First, assume that the Inequality (\ref{eqn 2.1})  holds, i.e., \begin{equation*}
	\|\widetilde{V}_k (\xi_{k})\|^2 \leq d_k \| (I_{E^{\ot {k-1}}} \ot D_{\widetilde{V}}) (\xi_{k})\|^2 + d \|(I_{E^{\ot {k-1}}} \ot \widetilde{V}^\dagger \widetilde{V}) (\xi_{k})\|^2 
	\end{equation*} for every $ \xi_{k} \in  E^{\ot k} \ot \mathcal{H}.$ Therefore it also holds for every $ \eta_{k}\in E^{\ot {k-1}}\ot N({\widetilde{V}})^\perp \subseteq E^{\ot k } \ot \mathcal{H},$ hence \begin{equation*}
	\|\widetilde{V}_k (\eta_{k})\|^2 \leq d_k \| (I_{E^{\ot {k-1}}} \ot D_{\widetilde{V}}) (\eta_{k})\|^2 +d \|(I_{E^{\ot {k-1}}} \ot \widetilde{V}^\dagger \widetilde{V}) (\eta_{k})\|^2. 
	\end{equation*} Since $\widetilde{V}^{\dagger}\widetilde{V}=P_{N({\widetilde{V})}^{\perp}},$  $I_{E^{\ot {k-1}}} \ot \widetilde{V}^{\dagger}\widetilde{V}=P_{E^{\ot {k-1}} \ot {N({\widetilde{V})}^{\perp}}},$ and $(I_{E^{\ot {k-1}}} \ot \widetilde{V}^\dagger \widetilde{V})(\eta_{k}) = \eta_{k} .$ Thus  \begin{align*}
	\|\widetilde{V}_k (\eta_{k})\|^2 & \leq d_k \| (I_{E^{\ot {k-1}}} \ot (\widetilde{V}^* \widetilde{V} -  \widetilde{V}^\dagger \widetilde{V})^\frac{1}{2}) (\eta_{k})\|^2 +d \| (\eta_{k})\|^2 \\& = d_k (\| (I_{E^{\ot {k-1}}} \ot \widetilde{V}) (\eta_{k})\|^2 - \| \eta_{k}\|^2 )+d \| \eta_{k}\|^2 ,
	\end{align*} which is the required inequality.

	On the other hand, suppose that the Inequality (\ref{2.2}) holds. Let $ \xi_{k} \in E^{\ot k} \ot \mathcal{H},$ then \begin{equation*}
	(I_{E^{\ot {k-1}}} \ot \widetilde{V}^\dagger \widetilde{V})(\xi_{k}) \in  E^{\ot{k-1}} \ot N({\widetilde{V}})^\perp.
	\end{equation*} By substituting $(I_{E^{\ot {k-1}}} \ot \widetilde{V}^\dagger \widetilde{V})(\xi_{k}) \in  E^{\ot{k-1}} \ot N({\widetilde{V}})^\perp$ for $ \eta_{k} $ in  Inequality (\ref{2.2}) and using $\widetilde{V}\widetilde{V}^\dagger\widetilde{V}=\widetilde{V},$ we can easily obtain Inequality (\ref{eqn 2.1}).
\end{proof}
\begin{lemma}\label{Lemma 2.4}
	Consider a CB-representation $(\sigma,V)$  of $E$ on $\mathcal{H}$ so that $\gamma({\widetilde{V}}) \geq 1 .$  Then for $ n\in \mathbb{N},$ \begin{equation*}
	\|h\|^2 = \sum_{i=0}^{n-1} \| (I_{E^{\ot i}} \ot P_{\mathcal{W}}) \widetilde{V}^{\dagger{(i)}}h\|^2 + \|(\widetilde{V}^\dagger)^nh\|^2 + \sum_{i=1}^{n} \| (I_{E^{\ot (i-1)}} \ot D_{\widetilde{V}}) \widetilde{V}^{\dagger{(i)}}h\|^2, \quad h \in \mathcal{H}
	\end{equation*} where $ P_{\mathcal{W}}:=I_{\mathcal{H}}-\widetilde{V}\widetilde{V}^\dagger .$
\end{lemma}
\begin{proof}
	We use the Mathematical induction. Let $ h \in \mathcal{H}, $ then \begin{equation*}
	\|(I_{\mathcal{H}}-\widetilde{V}\widetilde{V}^\dagger ) h\|^2=\|h\|^2-\|\widetilde{V}\widetilde{V}^\dagger h\|^2,
	\end{equation*} which is equivalent to \begin{equation}\label{2.4}
	\|h\|^2= \|(I_{\mathcal{H}}-\widetilde{V}\widetilde{V}^\dagger ) h\|^2 +\|\widetilde{V}\widetilde{V}^\dagger h\|^2. 
	\end{equation} Since $\|D_{\widetilde{V}}(\xi)\|^2= \|\widetilde{V}(\xi)\|^2- \|\widetilde{V}^\dagger\widetilde{V}(\xi)\|^2$ for every $ \xi \in E \ot \mathcal{H},$ replace $ \widetilde{V}^\dagger h $ by $ \xi ,$ we get \begin{equation*}
	\|D_{\widetilde{V}}(\widetilde{V}^\dagger h )\|^2= \|\widetilde{V}(\widetilde{V}^\dagger h )\|^2- \|\widetilde{V}^\dagger\widetilde{V}(\widetilde{V}^\dagger h )\|^2= \|\widetilde{V}(\widetilde{V}^\dagger h )\|^2- \|\widetilde{V}^\dagger h\|.
	\end{equation*} Therefore \begin{equation*}
	\|h\|^2= \|(I_{\mathcal{H}}-\widetilde{V}\widetilde{V}^\dagger ) h\|^2 + \|D_{\widetilde{V}}(\widetilde{V}^\dagger h )\|^2 + \|\widetilde{V}^\dagger h \|^2,
	\end{equation*} and thus \begin{equation*}
	\|h\|^2= \|P_{\mathcal{W}} h\|^2 + \|D_{\widetilde{V}}(\widetilde{V}^\dagger h )\|^2 + \|\widetilde{V}^\dagger h \|^2
	\end{equation*} which establishes the equality for $ n=1.$ Now suppose that the equality true for $ n.$ We observe that \begin{equation*}
	\|\widetilde{V}^{\dagger (n)} h\|^2= \| (I_{E^{\ot n}} \ot P_{\mathcal{W}}) \widetilde{V}^{\dagger( n)} h\|^2 + \|(I_{E^{\ot n}} \ot D_{\widetilde{V}})(\widetilde{V}^{\dagger (n+1)} h )\|^2 + \|\widetilde{V}^{\dagger( {n+1})} h\|^2 
	\end{equation*} for every $n \in \mathbb{N}.$ Using the induction hypothesis and the previous equation,
	we obtain \begin{align*}
	\|h\|^2 &= \sum_{i=0}^{n-1} \| (I_{E^{\ot i}} \ot P_{\mathcal{W}}) (\widetilde{V}^\dagger)^{(i)}h\|^2 + \|(\widetilde{V}^\dagger)^{(n)}h\|^2 + \sum_{i=1}^{n} \| (I_{E^{\ot (i-1)}} \ot D_{\widetilde{V}}) (\widetilde{V}^\dagger)^{(i)}h\|^2 \\& = \sum_{i=0}^{n-1} \| (I_{E^{\ot i}} \ot P_{\mathcal{W}}) (\widetilde{V}^\dagger)^{(i)}h\|^2 + \| (I_{E^{\ot n}} \ot P_{\mathcal{W}}) \widetilde{V}^{\dagger (n)} h\|^2 + \|(I_{E^{\ot (n-1)}} \ot D_{\widetilde{V}})(\widetilde{V}^{\dagger n} h )\|^2 \\& \:\:\:\:\:\:\ +  \|\widetilde{V}^{\dagger( {n+1}}) h\|^2 + \sum_{i=1}^{n} \| (I_{E^{\ot (i-1)}} \ot D_{\widetilde{V}}) (\widetilde{V}^\dagger)^{(i)}h\|^2 \\&  = \sum_{i=0}^{n} \| (I_{E^{\ot i}} \ot P_{\mathcal{W}}) (\widetilde{V}^\dagger)^{(i)}h\|^2 + \|(\widetilde{V}^\dagger)^{({n+1})}h\|^2 + \sum_{i=1}^{n+1} \| (I_{E^{\ot (i-1)}} \ot D_{\widetilde{V}}) (\widetilde{V}^\dagger)^{(i)}h\|^2.  \qedhere\end{align*}  \end{proof} \qedhere
The following lemma will help to prove the next theorem, which generalizes \cite[Lemma 3.3]{HV19}.
\begin{lemma}\label{Lemma 2.5}
	Consider $ n \in \mathbb{N}$ and a CB-representation  $(\sigma,V)$  of $E$ on $\mathcal{H}$ such that $\wV$ has closed range. Then  \begin{enumerate}
		\item For every $ h \in \mathcal{H},$ we have \begin{equation*}
		(I_{\mathcal{H}} -{\widetilde{V}}_n \widetilde{V}^{\dagger (n)})h=\sum_{i=0}^{n-1} \widetilde{V}_{i}(I_{E^{\ot i}}\ot P_{\mathcal{W}})\widetilde{V}^{\dagger (i)}h ;
		\end{equation*}   \item For $ 0\leq i\leq n $ and $\xi_n \in E^{\ot n }  \ot \mathcal{H},$ we get \begin{equation*}
		(I_{E^{\ot n} \ot \mathcal{H}} -\widetilde{V}^{\dagger (n)}{\widetilde{V}}_n) (\xi_n)=\sum_{i=0}^{n-1}(I_{E^{\ot n-i}} \ot \widetilde{V}^{\dagger (i)})(I_{E^{\ot n-(i+1)}}\ot P_{{\mathcal{W}^{\dagger}}})(I_{E^{\ot n-i}} \ot \widetilde{V}_i)(\xi_n),
		\end{equation*}  where $ P_{\mathcal{W}}=I_{\mathcal{H}}-\widetilde{V}\widetilde{V}^\dagger ,$ $ P_{\mathcal{W}^\dagger}:=I_{E\ot \mathcal{H}}-\widetilde{V}^\dagger\widetilde{V}. $
	\end{enumerate}
\end{lemma}
\begin{proof}
	$(1)$	Consider \begin{align*}
	I_{\mathcal{H}} -{\widetilde{V}}_n \widetilde{V}^{\dagger (n)}& = \sum_{i=0}^{n-1}\widetilde{V}_{i}\widetilde{V}^{\dagger (i)} -\widetilde{V}_{i+1}\widetilde{V}^{\dagger ({i+1})}  = \sum_{i=0}^{n-1}\widetilde{V}_{i}\widetilde{V}^{\dagger (i)} -\widetilde{V}_{i}(I_{E^{\ot i}} \ot \widetilde{V})(I_{E^{\ot i}} \ot \widetilde{V}^\dagger)\widetilde{V}^{\dagger ({i})}\\&= \sum_{i=0}^{n-1} \widetilde{V}_{i}(I_{E^{\ot i} \ot \mathcal{H}} - (I_{E^{\ot i}} \ot \widetilde{V})(I_{E^{\ot i}} \ot \widetilde{V}^\dagger))\widetilde{V}^{\dagger (i)} = \sum_{i=0}^{n-1} \widetilde{V}_{i}(I_{E^{\ot i} \ot \mathcal{H}} - (I_{E^{\ot i}} \ot \widetilde{V} \widetilde{V}^\dagger))\widetilde{V}^{\dagger (i)} \\&=  \sum_{i=0}^{n-1} \widetilde{V}_{i}(I_{E^{\ot i}}\ot P_{\mathcal{W}})\widetilde{V}^{\dagger (i)}.
	\end{align*}$(2)$ For second equality \begin{align*}
	I_{E^{\ot n} \ot \mathcal{H}} -\widetilde{V}^{\dagger (n)}{\widetilde{V}}_n &=\sum_{i=0}^{n-1} I_{E^{\ot n-(i+1)}} \ot (I_{E} \ot \widetilde{V}^{\dagger (i)}\widetilde{V}_i- \widetilde{V}^{\dagger ({i+1})}\widetilde{V}_{i+1} )\\&=\sum_{i=0}^{n-1} I_{E^{\ot n-(i+1)}} \ot (I_{E} \ot \widetilde{V}^{\dagger (i)}\widetilde{V}_i- (I_{E} \ot \widetilde{V}^{\dagger (i)})\widetilde{V}^{\dagger}\widetilde{V}(I_{E} \ot \widetilde{V}_{i})) \\& = \sum_{i=0}^{n-1}(I_{E^{\ot n-i}} \ot \widetilde{V}^{\dagger (i)})(I_{E^{\ot n-(i+1)}}\ot( I_{E\ot \mathcal{H}}-\widetilde{V}^\dagger\widetilde{V}))(I_{E^{\ot n-i}} \ot \widetilde{V}_{i})\\&= \sum_{i=0}^{n-1}I_{E^{\ot n-(i+1)}}\ot ((I_{E} \ot \widetilde{V}^{\dagger (i)}) P_{{\mathcal{W}^{\dagger}}}(I_{E} \ot \widetilde{V}_i)).
	\qedhere\end{align*}
\end{proof}

\begin{theorem}\label{2.1}
	Consider $ n \in \mathbb{N}$ and  a CB-representation  $(\sigma,V)$  of $E$ on $\mathcal{H}$ such that $\wV$ has closed range. Then  \begin{enumerate}
		\item $
		N({\widetilde{V}^{\dagger (n)}}) \subseteq \displaystyle\bigvee_{i=0}^{n-1}\{\widetilde{V}_i(\xi_{i}): \xi_{i} \in E^{\ot i} \ot \mathcal{W}\}.$
		\item If $(\sigma,V)$ is regular, we have
		$
		N({\widetilde{V}}_n)= \displaystyle\bigvee_{i=0}^{n-1} \{ (I_{E^{\ot n-i}}\ot {\widetilde{V}^{\dagger (i)}})(\xi_{n-i}):\xi_{n-i} \in E^{\ot n-(i+1)}\ot \mathcal{W}^{\dagger} \}.$ 
		
	\end{enumerate}  
\end{theorem} 
\begin{proof} \begin{enumerate}
		\item Let $ h \in N({\widetilde{V}^{\dagger (n)}}),$ then ${\widetilde{V}}_n\widetilde{V}^{\dagger (n)}h=0 .$ Therefore, by Lemma \ref{Lemma 2.5}, we conclude \begin{equation*}
		h = \sum_{i=0}^{n-1} \widetilde{V}_{i}(I_{E^{\ot i}}\ot P_{\mathcal{W}})\widetilde{V}^{\dagger (i)}h,
		\end{equation*} and thus $ h\in \displaystyle\bigvee_{i=0}^{n-1}\{\widetilde{V}_i(\xi_{i}): \xi_{i} \in E^{\ot i} \ot \mathcal{W}\}.$
		\item Let $ \xi_{n} \in N({{\widetilde{V}}_n}),$ then $\widetilde{V}^{\dagger (n)}{\widetilde{V}}_n (\xi_{n}) =0 .$ Therefore using Lemma \ref{Lemma 2.5} we have \begin{equation*}
		\xi_{n} = \sum_{i=0}^{n-1}(I_{E^{\ot n-i}} \ot \widetilde{V}^{\dagger (i)}) (I_{E^{\ot n-(i+1)}}\ot P_{{\mathcal{W}^{\dagger}}})(I_{E^{\ot n-i}} \ot \widetilde{V}_i) \xi_{n}
		\end{equation*}  and then $ \xi_{n} \in \displaystyle\bigvee_{i=0}^{n-1} \{ (I_{E^{\ot n-i}}\ot {\widetilde{V}^{\dagger (i)}})({\xi_{n-i}}): \xi_{n-i} \in E^{\ot n-(i+1)}\ot \mathcal{W}^{\dagger} \}.$ On the other hand, let $\xi_{n-i} \in E^{\ot n-(i+1)}\ot \mathcal{W}^{\dagger},$  for $ 0 \leq i \leq n-1 .$  Since $(\sigma,V)$ is regular,   $\mathcal{W}^\dagger =N({\widetilde{V}})\subseteq E \ot R^\infty({\widetilde{V}}) $ and  by using  Remark \ref{Inverse},  we have  $ (I_{E^{\ot({n-i})}} \ot \widetilde{V}_{i} \widetilde{V}^{\dagger (i)})(\xi_{n-i} )= \xi_{n-i} .$ Therefore \begin{align*}
		{\widetilde{V}}_n(I_{E^{\ot(n-i)}} \ot\widetilde{V}^{\dagger (i)} )(\xi_{n-i})=&\widetilde{V}_{n-i}(I_{E^{\ot(n-i)}} \ot\widetilde{V}_{ i})(I_{E^{\ot(n-i)}} \ot\widetilde{V}^{\dagger (i)} )(\xi_{n-i}) \\&=\widetilde{V}_{n-i}(I_{E^{\ot(n-i)}} \ot \widetilde{V}_{i}\widetilde{V}^{\dagger (i)} )(\xi_{n-i})\\&=\widetilde{V}_{n-i}(\xi_{n-i})=0,
		\end{align*} 
		that is, $(I_{E^{\ot(n-i)}} \ot\widetilde{V}^{\dagger (i)} )(\xi_{n-i}) \in N({\widetilde{V}}_{n}) $ for $ \xi_{n-i} \in E^{\ot n-(i+1)}\ot \mathcal{W}^\dagger , 0 \leq i \leq n-1 .$ Hence \begin{equation*}
		\displaystyle\bigvee_{i=0}^{n-1} \{ (I_{E^{\ot n-i}}\ot {\widetilde{V}^{\dagger (i)}})(\xi_{n-i})\: : \: \xi_{n-i} \in E^{\ot n-(i+1)} \ot \mathcal{W}^{\dagger}\}\subseteq N({\widetilde{V}}_{n}).
		\end{equation*}
		\qedhere\end{enumerate}
\end{proof}

Consider a CB-representation  $(\sigma,V)$  of $E$ on $\mathcal{H}$ and let  $\mathcal{W}$ be a wandering subspace for $(\sigma, V).$ In this section, we refer to the smallest $(\sigma, V)$-invariant subspace of $\mathcal{H}$ containing $\mathcal{W},$ by the notation $[\mathcal{W}]_{\wt{V}},$
i.e.,   $$[\mathcal{W}]_{\wt{V}}:=\displaystyle\bigvee_{n \in \mathbb{N}_0}\wt{V}_n(E^{\ot n} \ot \mathcal{W}).$$ Note that if $\mathcal{K}$ is $(\sigma, V)$-reducing and $\mathcal{W} \subseteq \mathcal{K}$, then $[\mathcal{W}]_{\wt{V}} \subseteq \mathcal{K}.$  If  $  R^{\infty}{({{V}})} $  reduces $(\sigma,V),$ then $ [\mathcal{W}]_{\wt{V}} \subseteq  R^{\infty}{({{V}})}^{\perp}, \:\mathcal{W}=\mathcal{H} \ominus \wt V(E \ot \mathcal{H}).$  Also using the equalities
\begin{equation*}
R^{\infty}({{V}})^\perp =\big( \bigcap_{n\in \mathbb{N}_0}  R({\widetilde{V}}_{n})\big)^\perp =\displaystyle\bigvee_{n \in \mathbb{N}_0} R({\widetilde{V}}_n)^\perp =\displaystyle\bigvee_{n \in \mathbb{N}_0} N({\widetilde{V}}_n ^*)  \:\: \mbox{and}
\end{equation*}  \begin{equation*}
R^{\infty}({{V}}^{\dagger*})^\perp =\big( \bigcap_{n\in \mathbb{N}_0}  R({\widetilde{V}}^{\dagger*{(n)}})\big)^\perp =\displaystyle\bigvee_{n \in \mathbb{N}_0} N({\widetilde{V}^{\dagger(n)}}), 
\end{equation*} we get the following duality relations.
\begin{corollary}\label{CC}
	Consider  a CB-representation  $(\sigma,V)$ of $E$ on $\mathcal{H}$ so that $\wV$ has closed range. Then   \begin{enumerate}
		\item $	R^{\infty}({{V}^{\dagger*}})^\perp \subset [ \mathcal{W}]_{\widetilde{V}},$ 
		\item if  $(\sigma,V)$ is regular, then $ R^{\infty}({V})^{\perp} = 
		\displaystyle\bigvee_{i=0}^{\infty} \{{\widetilde{V}^{\dagger*(i)}}(\xi_i )\: : \:\xi_i \in E^{\ot i} \ot \mathcal{W} \}.
		$
	\end{enumerate} 
	
\end{corollary}

\begin{theorem}\label{Theorem 2.9}
	Consider a regular CB-representation $(\sigma,V)$  of $E$ on $\mathcal{H}$ so that $  R^{\infty}{({{V}})} $  reduces $(\sigma,V),$ then $(\sigma,V)$ is bi-regular.
\end{theorem}
\begin{proof}
	Since $R(\wt{V})$ is closed and  $ \gamma({\widetilde{V}}^*) =\gamma({\widetilde{V}}),$  one can easily see that  $ R({\widetilde{V}}^\dagger) $ is closed. Now, let $ n\in \mathbb{N}$ and $ \xi \in E \ot  N({\widetilde{V}}^{\dagger(n)}), $ then by Theorem \ref{2.1}, we get  $ \xi \in E \ot R^{\infty}({{V}})^\perp.$  By using regularity of $(\sigma,V),$ $E \ot R^{\infty}({V})^\perp \subseteq N({\widetilde{V}})^\perp $ and hence $ \xi \in N({\widetilde{V}})^\perp =R({\widetilde{V}}^\dagger), $ which follows the required condition for bi-regularity, i.e., $ E \ot N({\widetilde{V}}^{\dagger(n)}) \subseteq R({\widetilde{V}}^\dagger) .$ 
\end{proof}
\begin{theorem}\label{Theorem 2.10}
	Consider a regular CB-representation $(\sigma,V)$ of $E$ on $\mathcal{H}$ so that   $\gamma({\widetilde{V}}) \geq 1 $ and satisfy the following inequality  \begin{equation}\label{2.7}
	\|\widetilde{V}_m (\xi_{m} )\|^2 \leq d_m \big(\| (I_{E^{\ot {m-1}}} \ot \widetilde{V}) (\xi_{m})\|^2 -\|(I_{E^{\ot {m-1}}} \ot \widetilde{V}^\dagger \widetilde{V}) (\xi_{m})\|^2\big) + \|(I_{E^{\ot {m-1}}} \ot \widetilde{V}^\dagger \widetilde{V}) (\xi_{m})\|^2,
	\end{equation} where $  \xi_{m} \in E^{\ot m} \ot \mathcal{H}, m\in \mathbb{N}$ and along with $\sum_{m \geq 2}$ $\frac{1}{d_m} = \infty.$ Then \begin{equation*}
	\mathcal{H} = 	[ \mathcal{W}]_{\widetilde{V}} +  R^{\infty}({{V}}),
	\end{equation*} where $ \mathcal{W} = \mathcal{H} \ominus\widetilde{V}(E \ot {\mathcal{H}}).$
\end{theorem}
\begin{proof} For $ m \in \mathbb{N},$   Inequality (\ref{2.7}) can be rewrite as \begin{equation*} \label{inq 2.8}
	\|\widetilde{V}_m (\xi_{m})\|^2 - \|(I_{E^{\ot {m-1}}} \ot \widetilde{V}^\dagger \widetilde{V}) (\xi_{m})\|^2  \leq  d_m \| (I_{E^{\ot {m-1}}} \ot D_{\widetilde{V}}) (\xi_{m})\|^2, \quad  \xi_{m} \in E^{\ot m}\ot \mathcal{H}. 
	\end{equation*}  Now replace  $ \xi_{m} $ by $  \widetilde{V}^{\dagger (m)}h, h \in \mathcal{H}, $ in the above inequality, we obtain
	\begin{align*}
	\| \widetilde{V}_m  \widetilde{V}^{\dagger (m)} {h} \|^2 -& \| (I_{E^{\ot {m-1}}} \ot \widetilde{V}^\dagger \widetilde{V})   \widetilde{V}^{\dagger (m)} {h} \|^2  \leq  d_m \| (I_{E^{\ot {m-1}}} \ot D_{\widetilde{V}})   \widetilde{V}^{\dagger (m)}{h} \|^2 
	\end{align*} and by using the identity $ \widetilde{V}^\dagger \widetilde{V}\widetilde{V}^{\dagger} =\widetilde{V}^{\dagger} ,$  \begin{align*}
	\| \widetilde{V}_m  \widetilde{V}^{\dagger (m)} {h} \|^2 -& \|    \widetilde{V}^{\dagger (m)} {h} \|^2  \leq  d_m \| (I_{E^{\ot {m-1}}} \ot D_{\widetilde{V}})   \widetilde{V}^{\dagger (m)}{h} \|^2 .
	\end{align*} Therefore \begin{align}\label{6.7}
	\sum_{m \geq 2 } \frac{	\| \widetilde{V}_m  \widetilde{V}^{\dagger (m)} {h} \|^2 - \|    \widetilde{V}^{\dagger (m)} {h} \|^2 }{d_m }& \leq \sum_{m\geq 2}    \| (I_{E^{\ot {m-1}}} \ot D_{\widetilde{V}})   \widetilde{V}^{\dagger (m)}{h} \|^2\\
	\nonumber &\leq  \|h\|^2,
	\end{align}  since $ (d_m)_{m\geq2} $ is positive sequence and the last inequality holds due to Lemma \ref{Lemma 2.4}.  Moreover, from the hypothesis $\displaystyle\sum_{m \geq 2}$ $\frac{1}{d_m} = \infty ,$ we conclude \begin{equation*}
	\liminf_m	\big\{	\| \widetilde{V}_m  \widetilde{V}^{\dagger (m)} {h} \|^2 -\|    \widetilde{V}^{\dagger (m)} {h} \|^2\big\} \leq 0.
	\end{equation*} Indeed, if $\liminf_m  \big\{	\| \widetilde{V}_m  \widetilde{V}^{\dagger (m)} {h} \|^2 -\|\widetilde{V}^{\dagger (m)} {h} \|^2\big\}  \geq \beta  \textgreater 0 ,$ then there is a $m_0 \in \mathbb{N} $ so that $\| \widetilde{V}_m  \widetilde{V}^{\dagger (m)} {h} \|^2 - \|    \widetilde{V}^{\dagger (m)} {h} \|^2    \geq \frac{\beta}{2} \textgreater 0 $ for $ m \geq m_0.$ Thus \begin{equation*}
	\sum_{m \geq 2 } \frac{	\| \widetilde{V}_m  \widetilde{V}^{\dagger (m)} {h} \|^2 - \|    \widetilde{V}^{\dagger (m)} {h} \|^2 }{d_m } \geq \sum_{m\geq 2} \frac{\beta}{ 2 d_m } = \infty,
	\end{equation*} which contradicts  the Inequality (\ref{6.7}). Since $\widetilde{V}^{\dagger}$ is  contraction, there exists a weakly convergent   subsequence $  \{\widetilde{V}_{m_j }\widetilde{V}^{\dagger (m_j)} h\} $ of $  \{\widetilde{V}_{m}\widetilde{V}^{\dagger (m)} h\}$ which converges weakly to $y$ for some $ y \in \mathcal{H}.$  In symbols : $  \widetilde{V}_{m_j }\widetilde{V}^{\dagger (m_j)} h  \rightharpoonup y .$ By the identity \begin{equation*}
	\widetilde{V}_{m}\widetilde{V}^{\dagger (m)}\widetilde{V}_{m_j}\widetilde{V}^{\dagger ({m_j})}h= 	 \widetilde{V}_{m_j}\widetilde{V}^{\dagger ({m_j})}h, \quad where\quad  m_j \geq m.\end{equation*}  Apply $m_j\to \infty $ to the above equation, 
	we get  $ \widetilde{V}_{m}\widetilde{V}^{\dagger (m)}y= y$ and hence $ y \in R(\widetilde{V}_{m}) ,$ for all $m \in \mathbb{N}.$ Now by Lemma \ref{Lemma 2.5}, $  h -  \widetilde{V}_{m_{j}}\widetilde{V}^{\dagger (m_{j})}h \in [\mathcal{W}]_{\widetilde{V}} .$ Since  $[\mathcal{W}]_{\widetilde{V}} $ is weakly  closed and $ h -  \widetilde{V}_{m_{j}}\widetilde{V}^{\dagger (m_{j})}h  \rightharpoonup h - y,$  $ h-y \in  [ \mathcal{W}]_{\widetilde{V}} .$ Consequently $  \mathcal{H}=[ \mathcal{W}]_{\widetilde{V}}+ R^{\infty}({V}) .$
\end{proof}
\begin{remark}
	Hereafter whenever we say that $(\sigma, V)$ satisfies the growth condition, it means it satisfies  (\ref{2.7}). 
\end{remark}
\begin{theorem}\label{TT}
	Consider a regular CB-representation $(\sigma,V)$ of $E$ on $\mathcal{H}$ so that  $\gamma({\widetilde{V}}) \geq 1 ,$ which satisfies the growth condition, then the following  hold :\begin{enumerate}
		\item $ R^{\infty}{({{V}})} $ reduces $(\sigma,V),$
		\item $(\widetilde{V}|_{E \ot {R^{\infty}({V})}})^{\dagger} = \widetilde{V}^{\dagger}|_{R^{\infty}({V})} = \widetilde{V}^{*}|_{R^{\infty}({V})} = (\widetilde{V}|_{E \ot {R^{\infty}({V})}})^{*},$
		\item  $(\sigma,V)$ is bi-regular,
		\item the restriction map 
		$	\widetilde{V}|_{({E \otimes R^{\infty}{({V})})\cap N({\widetilde{V}}})^\perp}:({E \otimes R^{\infty}{({V})})\cap N({\widetilde{V}})^{\perp}}\rightarrow R^\infty{({V})}$ is an unitary,
		\item $\mathcal{H}$ admits an orthogonal  decomposition, \begin{equation*}
		\mathcal{H} = [ \mathcal{W}]_{\widetilde{V}} \oplus R^{\infty}({{V}}).
		\end{equation*}
	\end{enumerate}
\end{theorem}
\begin{proof}
	(1) By using Theorem \ref{reg} and  Corollary \ref{invariant}, $ R^{\infty}{({{V}})} $  is $(\sigma, V)$-invariant  and $\widetilde{V}^{\dagger}R^{\infty}{({{V}})}  \subseteq E \ot R^{\infty}{({{V}})} ,$ therefore  we shall check that $R^{\infty}{({{V}})}^\perp$ is $(\sigma, V)$-invariant. Let $ h \in  R^{\infty}({V}) $ and  replace $ \xi_{m} \in E^{\ot m} \ot \mathcal{H} $ by $  \widetilde{V}^{\dagger (m)} h $ in the growth condition, we get 	\begin{align*}
	\| \widetilde{V}_m  \widetilde{V}^{\dagger (m)} {h} \|^2 -& \| (I_{E^{\ot {m-1}}} \ot \widetilde{V}^\dagger \widetilde{V})   \widetilde{V}^{\dagger (m)} {h} \|^2  \leq  d_m \| (I_{E^{\ot {m-1}}} \ot D_{\widetilde{V}})   \widetilde{V}^{\dagger (m)}{h} \|^2. 
	\end{align*} By using the identity $ \widetilde{V}^\dagger \widetilde{V}\widetilde{V}^{\dagger} =\widetilde{V}^{\dagger} ,$ 	\begin{align*}
	\| \widetilde{V}_m  \widetilde{V}^{\dagger (m)} {h} \|^2 -& \|    \widetilde{V}^{\dagger (m)} {h} \|^2  \leq  d_m \| (I_{E^{\ot {m-1}}} \ot D_{\widetilde{V}})   \widetilde{V}^{\dagger (m)}{h} \|^2. 
	\end{align*} Since $ h \in  R^{\infty}({V}) $  and Remark \ref{Inverse} follows that, $ \widetilde{V}_m \widetilde{V}^{\dagger (m)} h= h $ for every $m \in \mathbb{N}.$ Therefore \begin{align*}
	\|  {h} \|^2 - \|    \widetilde{V}^{\dagger (m)} {h} \|^2  \leq  d_m \| (I_{E^{\ot {m-1}}} \ot D_{\widetilde{V}})   \widetilde{V}^{\dagger (m)}{h} \|^2 .
	\end{align*}
	
	Thus \begin{equation*}
	\sum_{m \geq 2 } \frac{ \|  {h} \|^2 - \|    \widetilde{V}^{\dagger (m)} {h} \|^2}{d_m} \leq \sum_{m \geq 2 } \| (I_{E^{\ot {m-1}}} \ot D_{\widetilde{V}})   \widetilde{V}^{\dagger (m)}{h} \|^2  \leq \|  h \|^2,
	\end{equation*} where the last inequality follows from the Lemma \ref{Lemma 2.4}. By using the hypothesis $\displaystyle \sum_{m \geq 2 }\frac{1}{d_m} =\infty ,$ we conclude \begin{equation*}
	\liminf_m \{ \|  {h} \|^2 - \|    \widetilde{V}^{\dagger (m)} {h} \|^2\} \leq 0.
	\end{equation*} Since $\widetilde{V}^{\dagger}$ is contraction, $
	\lim_m \|  \widetilde{V}^{\dagger (m)}h\|= \|h\|$ and it follows that
	\begin{equation*} 
	\|{h}\|=\lim_m \|  \widetilde{V}^{\dagger (m)}h \| \leq  \|  \widetilde{V}^{\dagger}h\| \leq \| h\|,\:
	\end{equation*} hence $\|  \widetilde{V}^{\dagger}h\| = \| h\|,  h \in  R^{\infty}({V}).$ Therefore, for   $\eta \in E \ot R^{\infty}({V}) ,$  $\widetilde{V}(\eta ) \in R^{\infty}({V})$ and  \begin{equation*}
	\|D_{\widetilde{V}}(\eta)\|^2= \|\widetilde{V}(\eta)\|^2- \|\widetilde{V}^\dagger\widetilde{V}(\eta)\|^2 =0,
	\end{equation*}
	and thus \begin{equation}\label{xy}
	\widetilde{V}^*\widetilde{V}(\eta)=\widetilde{V}^{\dagger}\widetilde{V}(\eta), \quad \eta\in E\ot R^{\infty}({V}).
	\end{equation} 
	By  Corollary \ref{invariant}, 
	$ \widetilde{V}^{\dagger}(h) \in E \ot R^{\infty}({V}), $ for  $ h \in R^{\infty}({V})$  and  replace $ \eta $ by $\widetilde{V}^{\dagger}(h) $ in  Equation (\ref{xy}), we get \begin{equation*}
	\widetilde{V}^*\widetilde{V}\widetilde{V}^{\dagger}( h ) =\widetilde{V}^{\dagger}\widetilde{V}\widetilde{V}^{\dagger}(h). 
	\end{equation*} But $ \widetilde{V}^*\widetilde{V}\widetilde{V}^{\dagger}=\widetilde{V}^*$ and $\widetilde{V}^{\dagger}\widetilde{V}\widetilde{V}^{\dagger} = \widetilde{V}^{\dagger},$ therefore $\widetilde{V}^*h=\widetilde{V}^{\dagger}h \in E\ot R^{\infty}{({V}})$ for all $ h \in R^{\infty}{({V}}).$ This shows that $R^{\infty}({{V}})$ reduces $(\sigma, V).$ 
	
	(2) By the above observation, we conclude that $\widetilde{V}^{\dagger}|_{R^{\infty}({V})} = \widetilde{V}^{*}|_{R^{\infty}({V})}.$ Moreover,\begin{equation*}
	(\widetilde{V}|_{E \ot {R^{\infty}({V})}})^{*} = P_{E \ot {R^{\infty}({V})}} \widetilde{V}^*|_{{R^{\infty}({V})}} = \widetilde{V}^*|_{{R^{\infty}({V})}}
	\end{equation*} 
	and  using  Proposition \ref{dagger},  $(\widetilde{V}|_{E \ot {R^{\infty}({V})}})^{\dagger} = \widetilde{V}^{\dagger}|_{R^{\infty}({V})}.$ The statement (2) is proved.

	(3) Since $(\sigma,V)$ is  regular  and $ R^{\infty}{({{V}})} $ reduces $(\sigma,V),$ therefore  it follows  from Theorem \ref{Theorem 2.9}, $(\sigma,V)$ is bi-regular.

	(4) From Proposition \ref{invert}, the restriction map $	\widetilde{V}|_{(E \otimes R^{\infty}{({V})})\cap N({\widetilde{V}})^{\perp}}$ is bijective. So, we need to prove\\ $\widetilde{V}|_{(E \otimes R^{\infty}{({V})})\cap N({\widetilde{V}})^{\perp}}$ is an isometric. For this purpose, let $ \eta \in {E \otimes R^{\infty}{({V})}\cap N({\widetilde{V}})^{\perp}},$    $ \widetilde{V}^{\dagger}\widetilde{V}\eta=\eta ,$ and thus from  Equation (\ref{xy}), $\|\widetilde{V}\eta\|=\|\eta\|.$  Hence $	\widetilde{V}|_{E \otimes R^{\infty}{({V})}\cap N({\widetilde{V}})^{\perp}}$ is an unitary.

	(5) Since $ R^{\infty}{({{V}})} $ reduces $(\sigma,V),$ $[ \mathcal{W}]_{\widetilde{V}}\subseteq R^{\infty}({{V}}) ^\perp $  and using  Theorem \ref{Theorem 2.10}, we have $
	\mathcal{H} = [ \mathcal{W}]_{\widetilde{V}} \oplus R^{\infty}({{V}}).$
\end{proof}
\begin{remark}
	Observe that, from the above theorem, an orthogonal decomposition of $\mathcal{H}$ may not be unique.  Because of this, if we take   $(\sigma, V)$ as hyper-dagger, then the decomposition is unique.
	
	Since $(\sigma, V_n)$ is  CB-representation of $E^{\ot n}$ on $\mathcal{H},$ $R^{\infty}{({V})}=R^{\infty}{({V}_n)}$ and by Equation (\ref{xy}),  we get
	\begin{equation*}\label{ineq2.9}
	\widetilde{V}_n^*\widetilde{V}_n(\eta)=\widetilde{V}_n^{\dagger}\widetilde{V}_n(\eta)  \quad \mbox{and} \quad \widetilde{V}_n^{\dagger}\widetilde{V}_n\eta=\eta, \quad \: \: \eta \in {(E^{\ot n} \otimes R^{\infty}{({V})})\cap N({\widetilde{V}_n})^{\perp}}, n \in \mathbb{N},
	\end{equation*} thus $\|\widetilde{V}_n\eta\|=\|\eta\|.$ Using Proposition  \ref{invert} and $(\sigma, V_n)$ is $n$-dagger, $R({\widetilde{V}^{\dagger({n})}}) = N({\widetilde{V}_{n}})^{\perp}$ and therefore the map	$\widetilde{V}_{n}|_{({E^{\ot n} \otimes R^{\infty}{({V})})\cap N({\widetilde{V}_{n}}})^\perp}:({E^{\ot n} \otimes R^{\infty}{({V})})\cap N({\widetilde{V}_{n}})^{\perp}}\rightarrow R^\infty{({V})}$ 
	is unitary, for all $n \in \mathbb{N}.$

	To prove the uniqueness, suppose   $\mathcal{H} = \mathcal{H}_{1} \oplus \mathcal{H}_{2}$  is another decomposition of $\mathcal{H}$ into reducing subspaces  such that $(\sigma, V)|_{\mathcal{H}_{1}}$ has GWS-property   and   for $n \in \mathbb{N},$ \begin{align*}
	{\widetilde{V}}_n{|}_{({{E^{\otimes{n}} \otimes \mathcal{H}_{2})\cap N(\widetilde{V}_{n})^{\perp}}}}:{({{E^{\otimes{n}} \otimes \mathcal{H}_{2})\cap N(\widetilde{V}_{n})^{\perp}}}}\rightarrow \mathcal{H}_{2} 
	\end{align*} is unitary. Then  $\mathcal{H}_{1} = [\widetilde{W}]_{\widetilde{V}},$ where $ \widetilde{W} $ is wandering subspace for $(\sigma,V)|_{\mathcal{H}_{1}}.$
	Note that  $\widetilde{W} $ is uniquely determined by $\mathcal{H}_1,$ $ \widetilde{W} = N(\widetilde{V}^*|_{\mathcal{H}_{1}}),$ therefore 
	$\widetilde{W} \subseteq \mathcal{W}$ and hence $[\widetilde{W}]_{\widetilde{V}} \subseteq[ \mathcal{W}]_{\widetilde{V}}.$
	
	For each $ n \in \mathbb{N},$ {\begin{equation*}
		\mathcal{H}_{2}=\widetilde{V}_{n}(({{{E^{\otimes{n}} \otimes \mathcal{H}_{2})\cap N(\widetilde{V}_{n})^{\perp}}}}) \subseteq \widetilde{V}_{n}(E^{\ot n} \ot \mathcal{H}),
		\end{equation*}}
	clearly $\mathcal{H}_{2 } \subseteq R^{\infty}({{V}}).$ Moreover, \begin{equation*}
	\mathcal{H} = \mathcal{H}_{1} \oplus \mathcal{H}_{2} \subseteq [ \mathcal{W}]_{\widetilde{V}} \oplus R^{\infty}({{V}})=\mathcal{H},
	\end{equation*} which proves uniqueness.
\end{remark}
Let $(\sigma, V)$ be a CB-representation of $E$ on $\mathcal{H}$  so that $\widetilde{V}$ is expansive, i.e., $\| \xi\| \leq\| \wt{V}(\xi )\|,  \: \xi \in E \ot \mathcal{H}.$ Then  $N(\widetilde{V})=0$  and $ \gamma(\wV)\geq1,$  we get the following Wold-type decomposition.
\begin{corollary}
	Consider a  CB-representation $(\sigma,V)$ of $E$ on $\mathcal{H}$ so that $\widetilde{V}$ is expansive and satisfies the growth condition.  Then $\mathcal{H}$ has an unique orthogonal  decomposition, \begin{equation*}
	\mathcal{H} = [ \mathcal{W}]_{\widetilde{V}} \oplus R^{\infty}({{V}})
	\end{equation*} such that $(\sigma,V)|_{R^{\infty}{({V})}}$  is  isometric as well as fully  coisometric. That is, $(\sigma,V)$ admits Wold-type decomposition.
\end{corollary}

\begin{corollary}
	Consider a regular CB-representation $(\sigma,V)$ of $E$ on $\mathcal{H}$ so that  $\gamma({\widetilde{V}}) \geq 1 ,$ which satisfies the growth condition, then  \begin{equation*}
	[ \mathcal{W}]_{\widetilde{V}} = \displaystyle\bigvee_{n \in \mathbb{N}_0}N(({\widetilde{V}}^{\dagger})^{n}) =R^{\infty}({V}^{\dagger *})^{\perp}.
	\end{equation*}
\end{corollary}
\begin{proof}
	From statement (3) in Theorem \ref{TT},  $(\sigma,V)$ is bi-regular, and thus $\widetilde{V}^{\dagger *}$ is regular. By using  Corollary \ref{CC}, we have \begin{equation*}
	R^{\infty}({V})^{\perp} = \displaystyle\bigvee_{i=0}^{n-1} \big\{{\widetilde{V}^{\dagger*(i)}}(x_{i} )\: : \: x_{i} \in E^{\ot i} \ot \mathcal{W} \big\}.
	\end{equation*} Replace $\widetilde{V}$ by $\widetilde{V}^{\dagger *}$ in the above equation, and we get \begin{align*}
	R^{\infty}({V}^{\dagger *})^{\perp}	&= \displaystyle\bigvee_{i=0}^{n-1} \big\{{({\widetilde{V}^{\dagger *}})^{\dagger*(i)}}(x_{i} )\: : \: x_{i} \in E^{\ot i} \ot \mathcal{W} \big\}\\& = \displaystyle\bigvee_{i=0}^{n-1} \big\{ \widetilde{V}_{i}(x_{i})\: : \:  x_{i} \in E^{\ot i}\ot \mathcal{W} \big \}  = [ \mathcal{W} ]_{\widetilde{V}}.\qedhere
	\end{align*}
\end{proof}

Let $(\sigma,V)$ be a  CB-representation of $E$ on $\mathcal{H}$ and let $\mathcal{K}$ be a $(\sigma, V)$-invariant subspace of $\mathcal{H}.$ Then we get a wandering subspace   $N(\wV^*|_{\mathcal{K}}) =\mathcal{K}\ominus \wt{V}(E \ot \mathcal{H})$ for $(\sigma,V).$ On the other hand, if we start with a wandering subspace  $\mathcal{W}$  for $(\sigma,V),$ then $\mathcal{K}= \displaystyle\bigvee_{n \in \mathbb{N}_0}\wV_{n}(E^{\ot n}\ot \mathcal{W})$ is $(\sigma,V)$-invariant. In fact  $\mathcal{K}$ is the smallest  $(\sigma, V)$-invariant subspace which contains $\mathcal{W}.$ In this case $ \mathcal{W}=N(\wV^*|_{\mathcal{K}}).$ Indeed, \begin{align*}
N(\wV^*|_{\mathcal{K}})&=\mathcal{K}\ominus \wV(E \ot \mathcal{K})=\displaystyle\bigvee_{n \in \mathbb{N}_0}\wV_{n}(E^{\ot n}\ot \mathcal{W})\ominus \wV(E \ot \displaystyle\bigvee_{n\in \mathbb{N}_0}\wV_{n}(E^{\ot n}\ot \mathcal{W}))\\&= \displaystyle\bigvee_{n\in \mathbb{N}_0}\wV_{n}(E^{\ot n}\ot \mathcal{W})\ominus \displaystyle\bigvee_{n\in \mathbb{N}}\wV_{n}(E^{\ot n}\ot \mathcal{W})= \mathcal{W}.
\end{align*} Hence, an invariant subspace $\mathcal{K}$ determines the wandering subspace $\mathcal{W}$ in a unique way. This leads to the conclusion that there are one-to-one correspondences between the set of all $(\sigma, V)$-invariant subspaces of $\mathcal{H}$ and the set of all wandering subspaces of $\mathcal{H}$. This conclusion is the refinement of \cite[Theorem 2.4]{HV19} and \cite[Theorem 2.1]{O05}.
\begin{definition} Consider a CB-representation $(\sigma, V)$ of $E$ on $\mathcal{H}.$ If $$\bigcap_{n \in \mathbb{N}_{0}}\wt{V}_n(E^{\otimes n} \otimes \mathcal{H})=\{0\},$$ then we say $(\sigma, V)$    is   {\em analytic (pure)} (cf. \cite{HV19}).
\end{definition}
\begin{corollary}\label{invariant 2}
	Consider an  analytic, regular CB-representation $(\sigma,V)$ be  of $E$ on $\mathcal{H}$ such that  $\gamma({\widetilde{V}}) \geq 1 $ and satisfies  the growth condition. Let $\mathcal{K}$ be a $(\sigma,V)$-invariant subspace. Then, a wandering subspace $\mathcal{W}$ is present in such a way that \begin{equation*}
	\mathcal{K}=[\mathcal{W}]_{\widetilde{V}}.
	\end{equation*}
\end{corollary}
\begin{proof}
	Since $\mathcal{K}$ is $(\sigma,V)$-invariant subspace,   $(\sigma, V)|_{\mathcal{K}}$ is regular and satisfies the growth condition.  Then by Theorem \ref{TT}, $\mathcal{K}$ has an orthogonal decomposition
	\begin{equation*}
	\mathcal{K} = [ \mathcal{W}]_{\widetilde{V}} \oplus R^{\infty}({{V}|_{\mathcal{K}}}),
	\end{equation*}
	for some wandering subspace $\mathcal{W}$ and $R^{\infty}({{V}|_{\mathcal{K}}})=\displaystyle\bigcap_{n \in \mathbb{N}_{0}} \wt{V}_n(E^{\ot n} \ot \mathcal{K}),$ in fact $\mathcal{W}=N(\wV^*|_{\mathcal{K}}).$ Since $(\sigma, V)$ is analytic, $(\sigma, V)|_{\mathcal{K}}$ is also analytic, that is, $R^{\infty}({{V}|_{\mathcal{K}}})=0.$ Hence $\mathcal{K} = [ \mathcal{W}]_{\widetilde{V}}.$ \end{proof}

\begin{corollary}
	Consider a concave CB-representation $(\sigma, V)$ of $E$ on $\mathcal{H}.$ Then $(\sigma, V)$ admits Wold-type decomposition. That is, \begin{equation*}
	\mathcal{H} = [ \mathcal{W}]_{\widetilde{V}} \oplus R^{\infty}({{V}}).
	\end{equation*} such that  $(\sigma, V)|_{R^{\infty}({{V}})}$ is isometric and co-isometric.
\end{corollary} 
\begin{proof} By using the Inequality (\ref{CON}), 
	we obtain $$\frac{n-1}{n} \| \zeta\|^2 \leq\|(I_{E^{\otimes (n-1)}} \otimes \wt{V})(\zeta )\|^2, $$   for $\zeta \in E^{\ot n} \otimes \mathcal{H}, n \in \mathbb{N}$ and thus  $\frac{n-1}{n} \| \eta_{n-1} \ot \xi\|^2 \leq  \| \eta_{n-1} \ot \wt{V}(\xi)\|^2 \leq\|  \wt{V}(\xi )\|^2,  \xi \in E \ot \mathcal{H},\eta_{n-1} \in E^{\otimes n-1}  $ with $\|\eta_{n-1}\| \leq 1.$   It follows that  $\frac{n-1}{n} \|  \xi\|^2 \leq\|  \wt{V}(\xi )\|^2$ (See \cite[Lemma 2.2]{HV19}). This shows that  $\| \xi\| \leq\| \wt{V}(\xi )\|,  \: \xi \in E \ot \mathcal{H}.$   It follows that  $\widetilde{V}$ is left invertible,  MPI 	$\widetilde{V}^\dagger =  ({\widetilde{V}^*\widetilde{V}})^{-1}\widetilde{V}^*$ is the left inverse of $\widetilde{V},$ i.e., $\widetilde{V}^{\dagger}\widetilde{V}=I_{E \ot \mathcal{H}}.$ This makes it possible to write the Inequality (\ref{CON}) as \begin{equation*}
	\|\wt{V}_n(\zeta)\|^2 \leq \|(I_{E^{\ot (n-1)}}\ot \widetilde{V}^{\dagger}\widetilde{V}) \zeta\|^2 + n(\|(I_{E^{\otimes (n-1)}} \otimes \wt{V})(\zeta)\|^2-\|(I_{E^{\ot (n-1)}}\ot \widetilde{V}^{\dagger}\widetilde{V}) \zeta \|^2), \quad  \zeta \in E^{\otimes n}\ot \mathcal{H}
	\end{equation*}  and $ n \in \mathbb{N}.$ 	Note that $\| \xi\| \leq\| \wt{V}(\xi )\|,  \: \xi \in E \ot \mathcal{H},$   \begin{align*}
	\frac{\|\wV(\xi)\|}{\dis(\xi,N(\wV))}\geq\frac{\|\wV(\xi)\|}{ \|\xi\|} \geq 1,
	\end{align*} hence $ \gamma(\wV)\geq1.$ Thus by Theorem \ref{TT}, $(\sigma, V)$ admits  Wold-type decomposition.
\end{proof}
\section{ Contraction Intertwine With Left Invertible Covariant Representation }\label{section 7}
Bercovici, Douglas, and  Foias in \cite{BRC} proved that if a contraction $A$, defined on a Hilbert valued Hardy space $H^2_{\mathcal{E}}(\mathbb{D})$, commute with shift operator $S$ on $H^2_{\mathcal{E}}(\mathbb{D})$ and ${A^{*}}^{n}|_{{P_{S}}} \to 0$  strongly for $n \to \infty$  then  so is ${A^{*}}^{n} \to 0$  strongly for $n \to \infty ,$  where ${P_{S}}$  is projection on  $N(S^*).$  After that, S.Sarkar in \cite{SS21} proved the above result for the pure isometry case. At the beginning of this section, we define some notations and recall some definitions.

Consider a CB-representation $(\sigma, V)$  of $E$ on $\mathcal{H}$ and $\widetilde{V}$ has left inverse. Define $\wt{V}': E \otimes \mathcal{H} \longrightarrow \mathcal{H}$ by  $$\wt{V}':=\wt{V}(\wt{V}^*\wt{V})^{-1}.$$ 
It is easy to see that  $\wt{V}^*\wt{V}(\phi(a) \otimes I_{\mathcal{H}})=(\phi(a) \otimes I_{\mathcal{H}})\wt{V}^*\wt{V}$ for each $a \in \mathcal{B}$ and hence $\wt{V}'(\phi(a) \otimes I_{\mathcal{H}})=\sigma(a)\wt{V}',$ where $\phi$ is  the left action on $E.$ Then $(\sigma, V')$ is a CB-representation  of $E$ on $\mathcal{H},$  where $V': E \rightarrow B(\mathcal{H})$ is defined by $\wt{V}'(\xi \otimes h)=V'(\xi)h, \xi \in E, h \in \mathcal{H},$ and it is called  {\em Cauchy dual} (cf. \cite{HV19}) of $(\sigma, V).$

\begin{notation}
	$R^{\infty}(V'):=\displaystyle\bigcap_{n \in \mathbb{N}_{0}}\wt{V}'_n(E^{\otimes n} \otimes \mathcal{H}) \:\:\mbox{and}\:\:  \mathcal{W}':=\mbox{N} (\wt{V}'^*).$
\end{notation}
Observe that $\mathcal{W'}=\mbox{N}(\wt{V}^*)=\mathcal{W}.$
\begin{definition}
	Consider a  CB-representation $(\sigma, V)$ of ${E}$ on $\mathcal{H}.$ A bounded linear operator $A: \mathcal{H}\to \mathcal{H}$ is said to {\rm intertwine} $(\sigma, V)$ if it satisfies the  following 
	\begin{align*}
	AV(\xi)h=V(\xi)Ah \hspace{0.7cm}\mbox{and} \hspace{0.7cm}A \sigma(a)h= \sigma(a)Ah,
	\end{align*}
	where $\xi \in E, \: h \in \mathcal{H}$ and $ a \in \mathcal{B}.$
\end{definition}

To prove the primary conclusion of this section, we start with the lemma, which is quite helpful. 
\begin{lemma}\label{6.1}
	Consider a  CB-representation $(\sigma, V)$ of $E$ on $\mathcal{H}$ and $\widetilde{V}$ has left inverse. Assume that  $(\sigma,V)$ and its Cauchy dual $(\sigma,V')$ have GWS-property. Suppose $\mathcal{K}$ is a non-trivial proper $(\sigma,V)$-invariant subspace of $\mathcal{H}$ and $\mathcal{W}\subseteq \mathcal{K}^\perp,$ where $ \mathcal{W}= N(\widetilde{V}^*).$ Then there is a non-zero $h_1 \in\mathcal{K} $ so that $\widetilde{V}^*h_1 \in E \ot \mathcal{K}^\perp.$
\end{lemma}
\begin{proof}
	The subspace $\mathcal{W}\subseteq \mathcal{K}^{\perp} $ satisfies the following conditions:
	\begin{enumerate}
		\item $\mathcal{K}$ is not $(\sigma,V)$-reducing. Indeed, if $\mathcal{K}$ is $(\sigma,V)$-reducing, then $\mathcal{K} \perp  [\mathcal{W}]_{\wt{V}} =\mathcal{H} ,$ which contradicts to the assumption that $\mathcal{K}$ is proper.
		\item For any non-zero $ h \in \mathcal{K},$ $\widetilde{V}^{*}h \neq 0,$ otherwise $ h \in \mathcal{W}\subseteq \mathcal{K}^\perp.$   \end{enumerate}	Note that  $  N(\widetilde{V}^*)=  N(\widetilde{V}'^{*}),$  it follows from our assumptions that $
	\mathcal{H}= [\mathcal{W}]_{\wt{V}'}.$   Since $\mathcal{W}\subseteq \mathcal{K}^{\perp} $ and $\mathcal{K}$ is proper,  there exists $n \in \mathbb{N}$  such that  $\widetilde{V}'_{n}(E^{\otimes n} \otimes \mathcal{W})$ is not orthogonal to $\mathcal{K}.$ Then there exists a non-zero $\eta_n \in E^{\ot n} \ot \mathcal{W}$ such that  $\widetilde{V}'_{n}(\eta_n)$  is not orthogonal to $\mathcal{K}.$ Let $N$ be the smallest positive integer so that    $\widetilde{V}'_{N}(E^{\otimes N} \otimes \mathcal{W})$ is not orthogonal to $\mathcal{K}$ (such a $N$ always exists because $\mathcal{W}\subseteq \mathcal{K}^{\perp} $), that is,    $\widetilde{V}'_{l}(E^{\otimes l} \otimes \mathcal{W})$ is  orthogonal to $\mathcal{K}$ for all $0 \leq l \leq N-1.$ Choose the  non-zero elements  $h_1 \in \mathcal{K} $ and $h_2 \in \mathcal{K}^\perp$ such that  $\widetilde{V}'_{N}(\eta_N)=h_1 + h_2. $ Then by the observation $(2)$ and the minimality  on $N, $ we obtain
	$$	\widetilde{V}^* h_1  =\widetilde{V}^*\widetilde{V}'_{N}(\eta_N)-\widetilde{V}^*h_2 =\widetilde{V}^*\widetilde{V}'(I_{E} \ot \widetilde{V}'_{N-1})(\eta_{N})-\widetilde{V}^*h_2 =(I_{E} \ot \widetilde{V}'_{N-1})(\eta_{N})-\widetilde{V}^*h_2 \in E\ot \mathcal{K}^{\perp}.$$ This completes the proof. 
\end{proof}\qedhere

\begin{theorem}\label{pure1}
	Consider a  CB-representation $(\sigma, V)$ of $E$ on $\mathcal{H}$ and  $\widetilde{V}$ has left inverse. Suppose that $A$ is a bounded linear operator  on $\mathcal{H}$ which  intertwine $(\sigma,V).$ Assume that $(\sigma,V)$ and its Cauchy dual $(\sigma,V')$ have the GWS-property. Then the following assertions are equivalent:
	\begin{enumerate}
		\item $A$ is a pure contraction on $\mathcal{H}.$
		\item $ P_\mathcal{W} A|_{\mathcal{W}} $ is a pure contraction, where $ \mathcal{W}= N(\widetilde{V}^*).$
	\end{enumerate} 
\end{theorem}
\begin{proof} $(1)\implies (2):$
	Note that $\mathcal{W}$ is a $ A^*$-invariant subspace of $\mathcal{H},$  therefore for each $ n \in \mathbb{N},$ ${(P_\mathcal{W} A|_{\mathcal{W}})^{*}}^{n}={A^{*}}^{n}|_\mathcal{W}.$ Hence if $A$ is pure contraction, then $ P_\mathcal{W} A|_{\mathcal{W}} $ is also pure contraction.

	$(2)\implies (1):$ Suppose $ P_\mathcal{W} A|_\mathcal{W} $ is  pure contraction. Since $A$ is contraction, $\{A^n{A^{*}}^{n}\}_n $  is a decreasing sequence of non-negative operators. Therefore there is a positive operator, denoted by $D_A,$ such that the sequence  $\{A^n{A^{*}}^{n}\}_n $ converges to $D_A ^2$ in the strong operator topology.
	To prove  $A$ is a pure contraction, only we need to show that $N(D_A) =\mathcal{H}.$ Since $ P_\mathcal{W} A|_{\mathcal{W}} $ is a pure contraction,  $ \mathcal{W} \subseteq N(D_A).$ Indeed, for $ h \in \mathcal{W}$ \begin{align*}
	\|D_Ah \|=\lim_{n \rightarrow \infty}\|A^{*	n}h\|=\lim_{n \rightarrow \infty}\|{(P_\mathcal{W} A|_{\mathcal{W}})^{*}}^{n}h\|=0.
	\end{align*}
	Case 1: If $N(D_A)$ is $(\sigma,V)$-reducing  then \begin{align*}
	\mathcal{H}= [\mathcal{W}]_{\wt{V}} \subseteq N(D_A)\subseteq\mathcal{H},
	\end{align*} that is,  $N(D_A) =\mathcal{H}.$
	
	Case 2: If $N(D_A)$ is not   $(\sigma,V)$-reducing  then $N(D_A) $ is a proper subspace of $\mathcal{H}.$  Note that $ N(D_A)^{\perp} $ is  $(\sigma,V)$-invariant. Indeed, for  $ g \in N(D_{A})$ and using intertwine property of $A,$ we obtain \begin{align*}
	\|(I_{E}\ot D_{A})\widetilde{V}^{*}g\|^{2}=\lim_{n \rightarrow \infty}\langle\widetilde{V}(I_{E}\ot A^{n}{A^{*}}^{n})\widetilde{V}^{*}g,g\rangle=\lim_{n \rightarrow \infty}\langle A^{n}\widetilde{V}\widetilde{V}^{*}{A^{*}}^{n}g,g\rangle\leq\|\wt{V}\|^2 \|D_Ag\|^2 =0,
	\end{align*} thus $\wt{V}^*N(D_{A}) \subseteq E \otimes N(D_{A})$ and hence  $ N(D_A)^{\perp} $ is  $(\sigma,V)$-invariant. Therefore by  Lemma \ref{6.1}, there is a  $h_1 \in N(D_A)^{\perp} $ such that $\widetilde{V}^*h_1 \in E \ot N(D_A)$ and $h_{1}\neq 0.$ For each $m \in \mathbb{N},$ there exist $\xi^1_m  \in \mathcal{W}$ and $\xi^2_m \in \mathcal{W}^{\perp}$ such that \begin{equation*}\label{pure}
	A^{*m}h_1 = \xi^1_m +\xi^2_m.
	\end{equation*}
	Since $A$ is contraction, $\|\xi_m^2\| \leq \|h_1\|$ for all $m \in \mathbb{N}.$ Now apply $\widetilde{V}^*$ both the sides of the previous equation, we get 
	\begin{align*}
	\widetilde{V}^*A^{*m}h_1 =\widetilde{V}^*\xi^2_m.
	\end{align*} But $\widetilde{V}^*h_1 \in E \ot N(D_A),$ \begin{align*}
	\lim_{n \rightarrow \infty}\|\widetilde{V}^*\xi^2_m\|=\lim_{n \rightarrow \infty}\|\widetilde{V}^*A^{*m}h_1\|=\lim_{n \rightarrow \infty}\|(I_{E} \ot A^{*m})\widetilde{V}^*h_1\|=0.	\end{align*} Note that $\xi^2_m \in \mathcal{W}^\perp=\widetilde{V}(E\ot \mathcal{H}),$ there is $\eta^2_m \in E \ot \mathcal{H} $ such that $\widetilde{V} (\eta^2_m)=\xi^2_m$ and also $\widetilde{V}$ has  left inverse, hence $\widetilde{V}$ becomes bounded below.  Then there exists $ \alpha \textgreater 0$ such that \begin{align*}
	\|\eta^2_m\|\leq \alpha\|\widetilde{V} (\eta^2_m)\|=\alpha\|\xi^2_m\|.
	\end{align*}  From the above inequality,  we have \begin{align*}
	\lim_{m \rightarrow \infty}\|\xi^2_m\|^{2}=\lim_{m \rightarrow \infty}\|\widetilde{V} (\eta^2_m)\|^{2}\leq\lim_{m \rightarrow \infty}\|\widetilde{V}^*\widetilde{V} (\eta^2_m)\|\| (\eta^2_m)\|\leq\lim_{m \rightarrow \infty}\|\widetilde{V}^* (\xi^2_m)\|\| \alpha\xi^2_m\|= 0,  
	\end{align*} 
	since the sequence $\{\xi_m^2\}$ bounded by $h_1.$ Thus for each $n \in \mathbb{N},$ \begin{align*}
	\|D_Ah_1 \|=\|D_A{A^{*}}^{n}h_1\|=\|D_A(\xi^1_n +\xi^2_n)\|\leq 	\|D_A \|\|\xi^2_n\|.
	\end{align*}  If we let $n \rightarrow \infty$ form the above inequality, $D_Ah_1=0,$ i.e., $ h_1 \in N(D_A).$ However, by assumptions  $ h_1 \in N(D_A)^{\perp}$ and hence $h_1=0.$ This contradicts to our assumptions $N(D_A) $ is a proper subspace of $\mathcal{H}.$ Since $\{0\} \neq \mathcal{W} \subseteq N(D_A),$ the only possibility for $N(D_A)$ is $\mathcal{H}$ which is equivalent to showing that $A$ is  pure.
\end{proof}
\begin{corollary}
	Consider an analytic, regular  CB-representation $(\sigma, V)$ of $E$ on $\mathcal{H}$ so that $\gamma({\widetilde{V}}) \geq 1$ and satisfies the growth condition. Suppose $A$ is a bounded linear operator on $\mathcal{H}$ so that $A$ intertwine  $(\sigma,V).$ Then the following are equivalent:
	\begin{enumerate}
		\item $A$ is a pure contraction on $\mathcal{H}.$
		\item $ P_\mathcal{W} A|_{\mathcal{W}}$ is a pure contraction, where $ \mathcal{W} = N(\widetilde{V}^*).$
	\end{enumerate}
\end{corollary}
\begin{proof}
	By Theorem \ref{TT}, \begin{equation*}
	\mathcal{H} = [ \mathcal{W}]_{\widetilde{V}} \oplus R^{\infty}({{V}}).
	\end{equation*} Since $(\sigma,V)$ is pure and regular, $R^{\infty}({V}) =0$ and hence $\widetilde{V}$ is left invertible and $\widetilde{V}^{\dagger *} = \widetilde{V}'.$ Also note that $\mathcal{H}=\displaystyle\bigvee_{n \in \mathbb{N}_0}\widetilde{V}_{n}(E^{\ot n} \ot \mathcal{W}) =\displaystyle \bigvee_{n\in \mathbb{N}_0}\widetilde{V}^{\dagger *}_{n}(E^{\ot n} \ot \mathcal{W}) .$ Therefore by Theorem \ref{pure1}, $A$ is  pure contraction  if and only if $ P_\mathcal{W} A|_\mathcal{W}$ is  pure contraction.
\end{proof}\section{Applications to unilateral and bilateral weighted shifts}\label{7}
In this section, we apply the Wold decomposition to unilateral shift defined by Muhly and Solel in \cite{MS16} and an analog of bilateral shift used in \cite{EMZ15}. Here we introduce an analog of unilateral weighted shift on the vector-valued Hardy space, and we discuss GWS-property.
\subsection{$\mathbf{Unilateral\:\: shift}$}
In this subsection, we will consider  $\mathcal{B}$ to be $W^{*}$-algebra and $E$ to be $W^{*}$-correspondence over  $\mathcal{B}.$ Indeed, we shall use the setting of \cite{MS16}. The {\rm Fock space} of  a $W^*$-correspondence $E,$ defined as  $\mathcal{F}(E):=\displaystyle\bigoplus_{n \in \mathbb{N}_{0}}E^{\ot n},$  is a $W^{*}$-correspondence over $\mathcal{B}.$  The left module action of $\mathcal{B}$ on $\mathcal{F}(E)$ is denoted by $\phi^{\infty}$ and, is defined by $ \phi^{\infty}(a):=\displaystyle\oplus_{n \in \mathbb{N}_{0}}\phi^n(a),$  $ a \in \mathcal{B}.$   For $\xi\in E$, we define the  {\rm creation operator} $T_{\xi}$   on $\mathcal{F}(E)$ by
\begin{equation*}
T_{\xi}(\eta_{n})=\xi \ot \eta_{n},\:\:\:\:\: \eta_{n} \in E^{\ot n},\:n\in \mathbb{N}_{0}.
\end{equation*} 
A linear map $T :E \to \mathcal{L}(\mathcal{F}(E))$ is defined by $T(\xi)=T_{\xi},$ $\xi \in E,$  then it satisfies $ {T}(a\xi b)=\phi_{\infty}(a){T}(\xi)\phi_{\infty}(b), $ where $a,b\in \mathcal{B}$ and $\xi\in E.$  Note that,  the Fock module  $\mathcal{F}(E)$ is $ W^{*}$ -correspondence but not necessarily a Hilbert
space in general, so that  $(\phi_{\infty}, T)$ may not be covariant representation of $E$ on $\mathcal{L}(\mathcal{F}(E))$. But this can be overcome by composing the pair  $(\phi_{\infty}, T)$ with a faithful representation of the $W^{*}$-algebra $\mathcal{L}(\mathcal{F}(E))$ on a Hilbert space, this means that the pair $ (\psi\circ \phi_{\infty}, \psi\circ T) $ is a covariant representation of $E$ on $B(\mathcal{H})$ where  $\psi :\mathcal{L}(\mathcal{F}(E)) \to B(\mathcal{H}) $ is a faithful representation of the $C^{*}$-algebra $\mathcal{L}(\mathcal{F}(E))$ on a Hilbert space $\mathcal{H}.$

\begin{notation}
	Let	$\phi^n(\mathcal{B})^{c}$ be the set of all the elements  in $\mathcal{L}(E^{\ot n})$ which commute with  the image of $\mathcal{B}$ under $\phi^n, n \in \mathbb{N}_0.$  Similarly, the commutant of $\phi^{\infty}(\mathcal{B})$ in $\mathcal{L}(\mathcal{F}(E))$   denoted as $\phi^{\infty}(\mathcal{B})^{c}.$  Observe that $\phi^{0}(\mathcal{B})^{c}$ is simply the center of $\mathcal{B}.$ 
\end{notation}

Let $E$ be $W^{*}$-correspondence  over $\mathcal{B}.$ A  sequence  $\{Z_{k}\}_{k\in \mathbb{N}_{0} },$ where  $Z_{k}\in \phi^{k}(\mathcal{B})^{c} ,$ is called waight sequence  on $(\mathcal{B}, E)$ if  $Z_{0}=I_{\mathcal{B}}$  and $\sup\|Z_{k}\|<\infty.$  The following definition is due to Muhly-Solel \cite{MS16}:

\begin{definition} Let $E$ be a $W^{*}$-correspondence  over $\mathcal{B}$ and  $Z=\{Z_{k}\}_{k\in \mathbb{N}_{0} }$ be a  weight squence on $(E, \mathcal{B}).$  A linear map $W:E\to \mathcal{L}(\mathcal{F}(E))$ is called \rm{weighted shift} associated with the weight seqience $Z=\{Z_{k}\}_{k\in \mathbb{N}_{0} },$  if $W=DT$ where $D$ is the diagonal operator on $\mathcal{F}({E})$ corresponding to  $\{Z_{k}\}_{k\in \mathbb{N}_{0} },$ i.e., for $\eta_{n} \in E^{\ot n},$\begin{align*}
	W(\xi)\eta_{n}=DT(\xi)\eta_{n}=DT_{\xi}\eta_{n}=D(\xi\ot \eta_{n})=Z_{n+1}(\xi\ot \eta_{n}),\:\:\: \xi \in E. 
	\end{align*}	
	
\end{definition} 

Using the covariant condition of $T$  and $D$ lies in $\phi^{\infty}(\mathcal{B})^{c},$ one can easily see that $W$ satisfies  $ {W}(a\xi b)=\phi_{\infty}(a){W}(\xi)\phi_{\infty}(b), $ where $a,b\in \mathcal{B}$ and $\xi\in E.$  Suppose that  $\pi$ is a representation of $\mathcal{B}$ on a Hilbert space $\mathcal{H}.$ Define  a covariant representaion  $(\rho, S)$ of ${E}$ on  a Hilbert space $\mathcal{F}({E})\otimes \mathcal{H}$ with wight $Z$  by 
$$\rho(a)=\phi_{\infty}(a) \otimes I_{\mathcal{H}} \:\:\mbox{and}\:\:S(\xi)=  W(\xi) \ot I_{\mathcal{H}},  \:\:\:\: a \in \mathcal{B},\: \xi \in E . $$ 

Note that $S(\xi)(\eta)=( W(\xi) \ot I_{\mathcal{H} })\eta =(DT_{{\xi}} \otimes I_{\mathcal{H}})\eta=\sum_{n \in \mathbb{N}_{0}}(Z_{n+1}\ot I_{\mathcal{H}})(\xi \ot  \eta_{n}\ot h_{n} ). $  where $\eta=\oplus_{n \in \mathbb{N}_{0}}\eta_{n}\ot h_{n},$ $\xi \in E, \eta_n \in E^{\ot n}, h_n  \in \mathcal{H}$ and $D=[Z_0, Z_1,  Z_2, \cdots  ]$ is the diagonal operator on $\mathcal{F}({E}).$ We refer to $(\rho, S)$ as the weighted induced  representation induced by $\pi$ with weight $Z.$

Suppose that $Z_{n}$'s are surjective for each $ n \in \mathbb{N}_{0},$  $\text ran(\widetilde{S})=\bigoplus_{n \in \mathbb{N}}(Z_{n}\ot I_{\mathcal{H}})({	E^{\ot n}\ot  \mathcal{H}})= E \ot \mathcal{F}(E)\ot \mathcal{H}$ is closed and thus  $N(\widetilde{S}^{*})=\mathcal{H},$ that is, $\mathcal{H}$ is a  wandering subspace for $(\rho, S).$ In fact $\mathcal{H}$ is the generating wandering subspce for  $(\rho, S),$ since  $\widetilde{S}(E^{\ot n} \ot \mathcal{H})=E^{\ot n} \ot \mathcal{H}.$  

\begin{theorem}
	Let  ${(\rho, S)}$ be a weighted shift on $\mathcal{F}({E}) \ot \mathcal{H}$ with weight $\{Z_{k}\}_{k\in \mathbb{N}_{0} }$ such that  $Z_{k}$'s has closed range and injective for all $ k \in \mathbb{N}.$ Suppose that   $\gamma( Z_{k}\ot I_{\mathcal{H}})\geq 1, k \in \mathbb{N}$   and satisfies the following inequality 
	\begin{align}\label{11}&
	(Z_{k+n}(I_{E} \ot Z_{k+n-1}) \cdots(I_{E^{\ot k-1}} \ot Z_{n+1})\ot I_{\mathcal{H}})^{*} (Z_{k+n}(I_{E} \ot Z_{k+n-1}) \cdots(I_{E^{\ot k-1}} \ot Z_{n+1})\ot I_{\mathcal{H}}) \nonumber\\&\:\:\:\:\: -I_{E^{\ot k+n} \ot \mathcal{H}}  \leq d_{k}(I_{E^{\ot k-1}} \ot(Z_{n+1} \ot I_{\mathcal{H}})^{*}(Z_{n+1} \ot I_{\mathcal{H}}))-I_{E^{\ot k+n} \ot \mathcal{H}}\big), \: n \in \mathbb{N}
	\end{align}
	along with $\sum_{k\geq 2}$ $\frac{1}{d_k} = \infty,$  then ${(\rho, S)}$ admits GWS-property.
\end{theorem}
\begin{proof} 
	Since $Z_k$'s are injective,  $N(\widetilde{S})=  \{0\}$ and thus  the covariant representation $(\rho, S)$ is regular.  The Moore-Penrose inverse  of $\widetilde{S}, \:\widetilde{S}^{\dagger}:  \mathcal{F}(E) \ot \mathcal{H} \to  E \ot\mathcal{F}(E) \ot \mathcal{H}$ is 
	\begin{equation*}
	\widetilde{S}^{\dagger}(\eta)= \bigoplus_{{n \in \mathbb{N}}} ({Z}_{n}\ot I_{\mathcal{H}})^{^{\dagger}}(\eta_{n}\ot h_{n}), \:\:\: \eta =\displaystyle\oplus_{n \in \mathbb{N}_{0}}\eta_{n}\ot h_{n} \in  \mathcal{F}(E) \ot \mathcal{H}.
	\end{equation*}       
	
	Suppose if $\gamma(Z_{n}\ot I_{\mathcal{H}})\geq 1,$ for $n \in \mathbb{N},$    then by Proposition  \ref{Regular}  we get  $\|\widetilde{S}^{\dagger}(\eta)\|^{2}=\bigoplus_{{n \in \mathbb{N}}}\|({Z}_{n}\ot I_{\mathcal{H}})^{^{\dagger}})(\eta_{n}\ot h_{n})\|^{2}\leq\sum_{n \in \mathbb{N}_{0}} \|\eta_{n}\ot h_{n}\|^{2}.$ This demonstrates that $\widetilde{S}^{\dagger}$ is  contraction and hence $\gamma(\widetilde{S})\geq 1.$  Also, it is easy to see that $R^{\infty}(S) =\{0\}.$
	Now, we only need to show that $\widetilde{S}$ satisfies the growth  condition. For $\xi_1, \dots, \xi_{k} $ are in $E,\eta_{n}\in E^{\ot n}$ and $ h\in \mathcal{H},$ a straightforward calculation follows that \begin{align*}
	\widetilde{S}_{k}(\bigotimes_{q=1}^{k}\xi_q\ot  \eta_{n}\ot h)&=({W}_{k} \ot I_{\mathcal{H}})(\bigotimes_{q=1}^{n}\xi_q\ot \eta_{n}\otimes  h_{n})=(DT_{\xi_1} \otimes I_{\mathcal{H}})\cdots (DT_{\xi_k} \otimes I_{\mathcal{H}})(\eta_{n}\ot h_{n})\\&=((Z_{k+n}(I_{E} \ot Z_{k+n-1}) \dots(I_{E^{\ot k-1}} \ot Z_{n+1}))\ot I_{\mathcal{H}})(\bigotimes_{q=1}^{n}\xi_q\ot \eta_{n}\otimes  h_{n}). 
	\end{align*}  More generally,  for $\zeta\in E^{\ot k}$ 
	\begin{equation*}
	\widetilde{S}_{k}(\zeta\ot \oplus_{n \in \mathbb{N}_{0}} \eta_{n}\ot h_{n})=\displaystyle\sum_{n \in \mathbb{N}_{0}}((Z_{k+n}(I_{E} \ot Z_{k+n-1}) \dots(I_{E^{\ot k-1}} \ot Z_{n+1}))\ot I_{\mathcal{H}})(\zeta\ot \eta_{n}\otimes  h_{n}).
	\end{equation*} Since the weight sequence $\{Z_{k}\}_{k\in \mathbb{N}_{0} }$  satisfies the Inequality (\ref{11} ) and by using previous equation, we obtain \begin{align*}\label{gro1}
	\|\widetilde{S}_k (\zeta\ot\eta_{n}\ot h_{n} )\|^2 & \leq d_k \big(\| (I_{E^{\ot {k-1}}} \ot \widetilde{S}) (\zeta\ot\eta_{n}\ot h_{n})\|^2 -\|(I_{E^{\ot {k-1}}} \ot \widetilde{S}^\dagger \widetilde{S}) (\zeta\ot\eta_{n}\ot h_{n})\|^2\big)\\& \:\:\:\: + \|(I_{E^{\ot {k-1}}} \ot \widetilde{S}^\dagger \widetilde{S}) (\zeta\ot\eta_{n}\ot h_{n})\|^2
	\end{align*} along with $\sum_{k \geq 2}$ $\frac{1}{d_k} = \infty,$ where  $\zeta\in E^{\ot k}, \eta_{n}\in E^{\ot n}$ and $ h_{n}\in \mathcal{H}.$ Therefore  $(\rho,S)$ satisfies the growth condition (\ref{2.7}).  Hence by Theorem \ref{TT}, $(\rho, S)$ admits GWS-property.
	\qedhere \end{proof}

\subsection{$\mathbf{Bilateral\:\: Shift}$}

Fix $n \in \mathbb{N},$ $I_{n}=\{1,2,\dots,n\}.$ Let $\mathcal{H}$ be a separable  Hilbert space with an orthonormal basis $\{e_m \:\::\:\: m\in \mathbb{Z}\}$. Consider a bounded set of complex numbers $\{w_{i,m} \in \mathbb{C} : i \in I_n,\:\:m\in \mathbb{Z} \}.$ Define an $n$-tuple of bounded operator $V=(V_1, \dots, V_n)$  on  $\mathcal{H}$ by 
\begin{align}
V_{i}(e_{m})=w_{i,m}e_{i+nm},\:\:\: m \in \mathbb{Z},  \:\:i \in I_n.
\end{align}\label{7.2}
The  operator $V$ can be consider as a bounded operator   from $\bigoplus_{i=1}^n \mathcal{H}$ to $\mathcal{H}.$ Note that  $V=( V_1, V_2,\dots,V_n)$ is  non-commutative. Indeed, for distinct $i,j \in I_{n},$  $
V_i V_j (e_{m})= w_{j,m}w_{i,j+nm}e_{i+n(j+nm)}\neq w_{i,m}w_{j,i+nm}e_{j+n(i+nm)}= V_j V_i (e_{m}).
$  Suppose that for each $i \in I_n, w_{i,0}=0.$
Let $h \in \mathcal{H}$ and  $h =\sum_{ m\in \mathbb{Z}}a_{m}e_{m}, \: a_{m}\in \mathbb{C},$  we have
\begin{equation*}
V_{i}(h)=\sum_{ m\in \mathbb{Z}}a_{m}V_{i}(e_{m}) =\sum_{ m\in \mathbb{Z}\setminus\{0\}}a_{m}w_{i,m}e_{i+nm},\:\:\:\: 
\end{equation*} then $\text{ran}(V_{i})= \overline{\text{span}}\{e_{i+nm}\:\:\:\:|\:\:\: m \in \mathbb{Z}\setminus\{0\}\}.$ 	Since $\{i+nm \:\:|\:\: m\in \mathbb{Z}\} \cap \{j+nm \:\:|\:\: m\in \mathbb{Z}\}=\phi$ for distinct $i,j \in I_{n}$ and $ N(V_{i})=\text{span}\{e_{0}\},$  $\text{ran} (V_{i})\perp \text{ran} (V_{j})  $ and it follows that 
\begin{align*}
\text{ran}({V})=\bigoplus_{i=1}^{n}\text{ran}(V_{i})=\bigoplus_{i=1}^{n}\overline{\text{span}}\{e_{i+nm}\:\:|\:\:\: m \in \mathbb{Z}\setminus\{0\}\}\: \mbox{and }\:  N({V})= \bigoplus_{i=1}^{n}N(V_{i})=\bigoplus_{i=1}^{n}\text{span}\{e_{0}\}.
\end{align*}

Now for $l \geq 2$ and $i \in I_{n},$ we have  for  $h =\sum_{ m\in \mathbb{Z}}a_{m}e_{m} \in \mathcal{H},$
\begin{equation*}
V^{l}_{i}( h)=\sum_{ m\in \mathbb{Z}}a_{m}V^{l}_{i}(e_{m})=\sum_{m \in \mathbb{Z}\setminus A_{i,l}}a_{m}w_{i,m}(\prod_{j=2}^{l}w_{i,(\sum_{k=0}^{l-2}in^{k} +n^{j-1}m)})e_{(\sum_{k=0}^{l-1}i n^{k}+n^{l}m)},
\end{equation*}
where $m \in \mathbb{Z}\setminus A_{i,l}, \: A_{i,l}=\{0, \sum_{k=0}^{j-2}in^{k} \: \:| \:\: 2 \leq j \leq l  \}.$	
Therefore $ 
\text{ran}(V^{l}_{i})= \overline{\text{span}}\big\{e_{(\sum_{k=0}^{l-1}i n^{k}+n^{l}m)}\:\:\:\:|\:\:\:m\in \mathbb{Z}\setminus A_{i,l}   \big\}.$ Consider the set    $A_{i}=\bigcup_{l \in \mathbb{N}} A_{i,l},$ where $A_{i,0}:=\{0\},$ $i \in I_{n},$ then  $
R^{\infty}(V_{i})=\overline{\text{span}}\big\{e_{(\sum_{k=0}^{l-1}i n^{k}+n^{l}m)}\:\:\:\:|\:\:\:m\in \mathbb{Z}\setminus A_{i} , l \in \mathbb{N}  \big\}
$ and
\begin{align*}
R^{\infty}(V)=\bigoplus_{i=1}^{n}R^{\infty}(V_{i})=\bigoplus_{i=1}^{n}\bigcap_{l=0}^{\infty}\text{span}\big\{e_{(\sum_{k=0}^{l-1}i n^{k}+n^{l}m)}\:\:\:\:|\:\:\:m\in \mathbb{Z}\setminus A_{i}   \big\}.
\end{align*} For  each $i \in I_{n},$ $e_{0}\in \text{ran}(V_{n}^{l})$ for all $l \in \mathbb{N},$ then $e_{0}\in R^{\infty}(V_{i})$ and  it follows that  $N(V)=\bigoplus_{i=1}^{n}\text{span}\{e_{0}\}\subseteq \bigoplus_{i=1}^{n}R^{\infty}(V_{i}),$ this shows that    $ V$ is regular. Also, ${N(V)}^{\perp}=\bigoplus_{i=1}^{n}\mathcal{H}_{0},$ where $\mathcal{H}_{0}=\overline{\text{span}}\{e_{m}\:\:\:|\:\: m \in \mathbb{Z}\setminus\{0\}\}.$

Note that the map $V_{0}:\bigoplus_{i=1}^{n}\mathcal{H}_{0}\to \bigoplus_{i=1}^{n}\text{ran}V_{i}$ defined by  $V_{0}|_{\mathcal{H}_{0}}=V_{i}|_{\mathcal{H}_{0}}: \mathcal{H}_{0} \to \text{ran}V_{i}$ where $\mathcal{H}_{0}$  is at $i^{th}$ position, is invertible.
Let $e_{i+nm}\in \text{ran}V_{i},\: m \neq 0,$  $V_{i}^{-1}(e_{i+nm})=\frac{1}{w_{i,m}}e_{m} \in \mathcal{H}_{0},$ and also $\text{ran}(V_i)$'s are closed and orthogonal to each other.  Therefore, the Moore-Penrose  inverse $V^{\dagger}: \mathcal{H} \to \bigoplus_{i=1}^{n}\mathcal{H},$ 
is of the form $V^{\dagger}=(V_{1}^{\dagger},\dots,V_{n}^{\dagger}),$ where  $ V_i^{\dagger}$ is Moore-Penrose inverse of $V_i.$ 
More generally, for  $h=\bigoplus_{i=1}^{n} h_{i}, h_{i}\in \text{ran}(V_{i}),$ then $h_{i}=\sum_{m \in \mathbb{Z}\setminus\{0\}}\alpha_{i+nm}e_{i+nm},$ for  some $\alpha_{i+nm}\in \mathbb{C}$ and  it yields that \begin{align*}\label{MMP}
V^{\dagger}h=V_{0}^{-1}h=\big(\sum_{m \in \mathbb{Z}\setminus\{0\}}\frac{\alpha_{1+nm}}{w_{1,m}}e_{m},\sum_{m \in \mathbb{Z}\setminus\{0\}}\frac{\alpha_{2+nm}}{w_{2,m}}e_{m},\dots, \sum_{m \in \mathbb{Z}\setminus\{0\}}\frac{\alpha_{n+nm}}{w_{n,m}}e_{m}\big).
\end{align*}

Suppose that   $|w_{i, m}| \geq 1$ for all $i \in I_n, m \in \mathbb{Z}\setminus \{0\}.$ Let $ h' =h +h_0\in  {\text{ran}(V)} \oplus{\text{ran}(V)}^{\perp},$ using the previous equation, we get   $
\|V^{\dagger}h'\|^{2}=\sum_{i=1}^{n}\sum_{m\in \mathbb{Z}\setminus\{0\}}{|\frac{\alpha_{i+nm}}{w_{i,m}}|}^{2}\leq \sum_{i=1}^{n}\sum_{m\in \mathbb{Z}\setminus\{0\}}{|{\alpha_{i+nm}}|}^{2}\leq \|h\|^{2}\leq \|h'\|^{2}.
$   This shows that $V^{\dagger}$ is  contraction, and thus  Proposition \ref{Regular}  implies that $\gamma(V)\geq 1.$

Let $E$ be an $n$-dimensional Hilbert space  with an    orthonormal basis $\{\delta_i\}_{i\in I_{n}}.$  
Define a completely bounded covariant representation  $(\rho, S^{w})$ of  $E$ on $\mathcal{H}$  by
$\rho(a)=a I_{\mathcal{H}} \:\: \mbox{and}\: \:\: S^{w}(\delta_i)=V_i,  \:\:a \in \mathbb{C}, \: i \in I_n,$
where   $(V_1, \dots, V_n)$ is in (\ref{7.2}). The  representation  $(\rho, S^{w})$  is   called  {\it bilateral weighted shift} with weight $\{w_{i,m} : 1\leq i \leq n,\:\:\:m\in \mathbb{Z} \}.$ 


Let $l \geq 2$ and $i \in I_{n},$ 
we can deduce that for $h =\sum_{ m\in \mathbb{Z}}a_{m}e_{m} \in \mathcal{H}, a_{m} \in\mathbb{C},$ \begin{equation*}
{S^{w}(\delta_i)}^{l} h=\sum_{ m\in \mathbb{Z}}a_{m}{S^{w}(\delta_i)}^{l}(e_{m})=\sum_{m \in \mathbb{Z}\setminus A_{i,l}}a_{m}w_{i,m}(\prod_{j=2}^{l}w_{i,(\sum_{k=0}^{l-2}in^{k} +n^{j-1}m)})e_{(\sum_{k=0}^{l-1}i n^{k}+n^{l}m)},
\end{equation*} 	where $  A_{i,l}=\{0, \sum_{k=0}^{j-2}in^{k} \: \:| \:\: 2 \leq j \leq l  \}$ and it follows that  $
\text{ran} {S^{w}(\delta_i)}^{l}= {\overline{\text{span}}}\big\{e_{\sum_{k=0}^{l-1}i n^{k}+n^{l}m}\:\:|\:m\in \mathbb{Z}\setminus A_{i,l}   \big\}.$  
For $i \in I_{n},$ consider the set  $A_{i}=\bigcup_{l\in \mathbb{N}} A_{i,l}$   then $
R^{\infty}({S^{w}(\delta_i)})=\bigcap_{l=0}^{\infty}{\overline{\text{span}}}\big\{e_{\sum_{k=0}^{l-1}i n^{k}+n^{l}m}\:\:\:\::\:\:\:m\in \mathbb{Z}\setminus A_{i}   \big\}$
and thus
\begin{align*}
R^{\infty}(S^{w})=\bigoplus_{i=1}^{n}R^{\infty}(S^{w}_{i})=\bigoplus_{i=1}^{n}\bigcap_{l=0}^{\infty}{\overline{\text{span}}}\big\{e_{\sum_{k=0}^{l-1}i n^{k}+n^{l}m}\:\:\:\::\:\:\:m\in \mathbb{Z}\setminus A_{i}   \big\}.
\end{align*} 
Note that for $i \in I_{n},$ $e_{0}\in \text{ran}({S^{w}}(\delta_i)^{l}),$ and thus $e_{0}\in R^{\infty}(S^{w}_{n}).$ It follows that  $N(\widetilde{S^{w}})=\bigoplus_{i=1}^{n}\text{span}\{\delta_i \ot e_{o}\}\subseteq E \ot \bigoplus_{i=1}^{n}R^{\infty}(S^{w}_{i})$,  that is,   $(\rho, S^w)$ is regular.  Furthermore, based on the previous obeservation for $V,$  $\widetilde{S}^{w\dagger}$ is  contraction,  and thus   $\gamma(\widetilde{S}^{w})\geq 1.$  Next,  we must determine  under what condition  on the set $\{w_{i,m} : 1\leq i \leq n,\:\:\:m\in \mathbb{Z} \},$    $(\rho, S^{w})$ satisfies the growth condition (\ref{2.7}).

\begin{theorem}\label{SSS}   
	Let $(\rho, S^{w})$ be a  bilateral weighted shift with weight $\{w_{i,m} : \: i \in I_n,\:m\in \mathbb{Z} \}$ on a separable Hilbert space $\mathcal{H}$  such that  \begin{enumerate}
		\item [(i)] for each $i\in I_{n},$  $w_{i,m}= 1$ for all $m <0,$  $w_{i,0}=0$ and $w_{i,m}\geq 1$ for every $m>0.$ 
		\item [(ii)]
		$\big(w_{i_{k},m}^{2}\prod_{q=0}^{k-1}w_{i_{k-q},n_{k-q+1,k}}^{2}\big)-1\leq d_{k}\big({w_{i_{k},m}^{2}}-1\big),$ where $(d_{k})_{k \in \mathbb{N}}$ is positive sequence along with $\sum_{k \geq 2}$ $\frac{1}{d_k} = \infty$ and  $n_{p,k}=\sum_{l=p}^{k}n^{l-p}i_{l}+n^{k-p+1}m.$
	\end{enumerate} Then 
	\begin{enumerate}
		\item $(\rho, S^{w})$ satisfies the assumption of Theorem \ref{TT}, 
		\item $S^{w\dagger}$ is a regular contraction,
		\item  	 the subspace  $\bigvee_{i=1}^{n}\big\{\big\{e_{\sum_{k=0}^{l-1}i n^{k}+n^{l}m}\::\:m\in \mathbb{Z}\setminus A_{i}, l \in \mathbb{N}   \big\}  \setminus \{e_0\}\big\}$
		reduces  $(\rho, S^{w})$ and  the restriction of $(\rho, {S^{w}})$  on this subspace   is isometric and fully co-isometric.
		
	\end{enumerate} 
\end{theorem}

\begin{proof}
	
	Since $w_{i,m}= 1$ for all $m <0,$   $w_{i_{k},m}^{2}\prod_{q=0}^{k-1}w_{i_{k-q},n_{k-q+1,k}}^{2}$ is either   $1$ or $0$ for all $m<0$ and $ k\geq 2.$ Then, we get 
	\begin{equation}\label{PPP} 
	\big(w_{i_{k},m}^{2}\prod_{q=0}^{k-1}w_{i_{k-q},n_{k-q+1,k}}^{2}\big)-1\leq d_{k}\big({w_{i_{k},m}^{2}}-1\big),\:\:\:\:   k \geq 2 , m\in\mathbb{Z}\setminus\{0\}.
	\end{equation}
	For  $k\in \mathbb{N},$ $
	E^{\ot k} \ot \mathcal{H}=\overline{\text{span}}\{\bigotimes_{q=1}^{k}\delta_{i_{q}}\ot e_{m} \:\:\:: 1\leq {i_{1}},{i_{2}},\dots,{i_{k}}\leq n,\:\: m \in \mathbb{Z} \},$ by the previous discussion
	\begin{align*}
	\widetilde{S}^{w}_{k}(\eta_k) =\displaystyle\sum_{1\leq {i_{1}},{i_{2}},\dots,{i_{k}}\leq n,\:\: m \in \mathbb{Z}\setminus\{0\}} \alpha_{ {i_{1}},{i_{2}},\dots,{i_{k}},m}w_{i_{k},m}\big(\prod_{q=0}^{k-1}w_{i_{k-q},n_{k-q+1,k}}e_{n_{1,k}}\big),
	\end{align*} where  $\eta_k=\displaystyle\sum_{1\leq {i_{1}},{i_{2}},\dots,{i_{k}}\leq n,\:\: m \in \mathbb{Z}} \alpha_{ {i_{1}},{i_{2}},\dots,{i_{k}},m}\bigotimes_{q=1}^{k}\delta_{i_{q}}\ot e_{m}\in E^{\ot k}\ot \mathcal{H}$ and $n_{p,k}=\sum_{l=p}^{k}n^{l-p}i_{l}+n^{k-p+1}m.$  Then
	\begin{align*}
	\|\widetilde{S}^{w}_{k}(\eta_k)\|^{2}&=\displaystyle\sum_{1\leq {i_{1}},{i_{2}},\dots,{i_{k}}\leq n,\:\: m \in \mathbb{Z}\setminus\{0\}} |{\alpha_{ {i_{1}},{i_{2}},\dots,{i_{k}},m}|^{2}w_{i_{k},m}^{2}}\prod_{q=0}^{k-1}w_{i_{k-q},n_{k-q+1,k}}^{2}.
	\end{align*}  
	Also, note that 
	\begin{align*}
	(I_{E^{\ot {k-1}}}\ot \widetilde{S}^{w})(\eta_k)&=\displaystyle\sum_{1\leq {i_{1}},{i_{2}},\dots,{i_{k}}\leq n,\:\: m \in \mathbb{Z}} \alpha_{ {i_{1}},{i_{2}},\dots,{i_{k}},m}\bigotimes_{q=1}^{k-1}\delta_{i_{q}} \ot \widetilde{S}^{w}(\delta_{i_{k}} \ot e_{m})\\&=\displaystyle\sum_{1\leq {i_{1}},{i_{2}},\dots,{i_{k}}\leq n,\:\: m \in \mathbb{Z}} \alpha_{ {i_{1}},{i_{2}},\dots,{i_{k}},m}\bigotimes_{q=1}^{k-1}\delta_{i_{q}} \ot {S}^{w}_{i_{k}}( e_{m})\\&=\displaystyle\sum_{1\leq {i_{1}},{i_{2}},\dots,{i_{k}}\leq n,\:\: m \in \mathbb{Z}\setminus\{0\}} \alpha_{ {i_{1}},{i_{2}},\dots,{i_{k}},m}\bigotimes_{q=1}^{k-1}\delta_{i_{q}}\ot (w_{i_{k},m} e_{i_{k}+nm}),
	\end{align*} where $ m \in \mathbb{Z}\setminus\{0\}$ and $ \|(I_{E^{\ot {k-1}}}\ot \widetilde{S}^{w})(\eta_k)\|^{2}=\displaystyle\sum_{1\leq {i_{1}},{i_{2}},\dots,{i_{k}}\leq n,\:\: m \in \mathbb{Z}\setminus\{0\}}|{\alpha_{ {i_{1}},{i_{2}},\dots,{i_{k}},m}}|^{2}w_{i_{k},m}^{2}.$
	Similarly
	\begin{align*}
	(I_{E^{\ot {k-1}}}\ot \widetilde{S}^{w\dagger}\widetilde{S}^{w})(\eta_k)&=\displaystyle\sum_{1\leq {i_{1}},{i_{2}},\dots,{i_{k}}\leq n,\:\: m \in \mathbb{Z}\setminus\{0\}} \alpha_{ {i_{1}},{i_{2}},\dots,{i_{k}},m}\bigotimes_{q=1}^{k-1}\delta_{i_{q}} \ot \widetilde{S}^{w\dagger}(w_{i_{k},m} e_{i_{k}+nm})\\&=\displaystyle\sum_{1\leq {i_{1}},{i_{2}},\dots,{i_{k}}\leq n,\:\: m \in \mathbb{Z}\setminus\{0\}} \alpha_{ {i_{1}},{i_{2}},\dots,{i_{k}},m}\bigotimes_{q=1}^{k-1}\delta_{i_{q}} \ot (\delta_{i_{k}} \ot e_{m}).
	\end{align*}
	This  implies that $
	\|(I_{E^{\ot {k-1}}}\ot \widetilde{S}^{w\dagger}\widetilde{S}^{w})(\eta_k)\|^{2}=\displaystyle\sum_{1\leq {i_{1}},{i_{2}},\dots,{i_{k}}\leq n,\:\: m \in \mathbb{Z}\setminus\{0\}} |{\alpha_{ {i_{1}},{i_{2}},\dots,{i_{k}},m}}|^{2}.$

	Finally, it follows from the above and   by  using the Equation (\ref{PPP}), we obtain
	\begin{align*}
	\|\widetilde{S}^{w}_k (\eta_k)\|^2 \leq d_k \big(\| (I_{E^{\ot {k-1}}} \ot \widetilde{S}^{w}) (\eta_k)\|^2 -\|(I_{E^{\ot {k-1}}} \ot \widetilde{S}^{w\dagger} \widetilde{S}^{w}) (\eta_k)\|^2\big) + \|(I_{E^{\ot {k-1}}} \ot \widetilde{S}^{w\dagger} \widetilde{S}^{w}) (\eta_k)\|^2,
	\end{align*} which means $(\rho,S^{w})$ satisfies the growth condition and  it satisfies  the hypothesis of Theorem \ref{TT}.
	Since $ R^{\infty}(S^{w})$ and $ {\overline{\text{span}}} \{e_m  \::\: m\in \mathbb{Z} \setminus \{0\} \}$  are  both $(\rho, S^w)$ reducing subspaces,  $ R^{\infty}(S^{w}) \cap {\overline{\text{span}}} \{e_m \::\: m\in \mathbb{Z} \setminus \{0\} \}=\bigvee_{i=1}^{n}\big\{\big\{e_{\sum_{k=0}^{l-1}i n^{k}+n^{l}m}\::\:m\in \mathbb{Z}\setminus A_{i} , l \in \mathbb{N}  \big\}  \setminus \{e_0\}\big\}$ is also $(\rho, S^w)$ reducing. Observe that  $E \ot R^{\infty}(S^{w}))\cap N(\widetilde{S^{w}})^{\perp}=E \ot \bigvee_{i=1}^{n}\big\{\big\{e_{\sum_{k=0}^{l-1}i n^{k}+n^{l}m}\::\:m\in \mathbb{Z}\setminus A_{i}, l \in \mathbb{N}   \big\}  \setminus \{e_0\}\big\}$ and by using  part $5$ of Theorem \ref{TT}, the restriction of $(\rho, {S^{w}})$  on $\bigvee_{i=1}^{n}\big\{\big\{e_{\sum_{k=0}^{l-1}i n^{k}+n^{l}m}\::\:m\in \mathbb{Z}\setminus A_{i},l \in \mathbb{N}   \big\}  \setminus \{e_0\}\big\}$   is isometric and fully co-isometric. This proves  part 3 of the theorem.
\end{proof}
\begin{remark}
	Let $k \in \mathbb{N},$ define a set $\Gamma(k, I_n)=\{f \: | \: f: I_k \to I_n  \}$ and $\Gamma(0, I_n)=\{0\}.$   Consider the Hilbert space \begin{align*}
	\bigoplus_{f \in \Gamma(k,I_{n})}\mathcal{H}=\bigoplus_{j=1}^{n^{k}}\mathcal{H}=\bigoplus_{j=1}^{n^{k-1}}\big(\bigoplus_{i=1}^{n}\mathcal{H}\big)
	\end{align*} with an orthonormal basis $\{e_{m}^{(f,k)}\:\:|\:\:\ m \in \mathbb{Z},\:\: f \in \Gamma(k,I_{n})\},$ where $e_{m}^{(f,k)}:=e_m$ is the $f$ position in the Hilbert space $\bigoplus_{f \in \Gamma(k,I_{n})}\mathcal{H}.$  For  $k \in \mathbb{N}$ and $f \in \Gamma(k, I_n),$  define $V_f=V_{f(1)}V_{f(2)}\cdots V_{f(k)},$ where $V_0=I_{\mathcal{H}},$ and $V^k=(V_f)_{f \in\Gamma(k, I_n) }$ be a $|\Gamma(k, I_n)|$-tuple of bounded operators on the Hilbert space $\mathcal{H}.$ The operator $V^k$ can be consider as a bounded operator  from  $\bigoplus_{f \in \Gamma(k,I_{n})}\mathcal{H}$ to $\mathcal{H}.$
	
	Since   $ \text{ran}V_{i}$ is orthogonal to each other, $V^{\dagger}V=(V_1^{\dagger}V_1, V_2^{\dagger}V_2, \dots, V_n^{\dagger}V_n).$ For $k \in \mathbb{N},$ define $[V^{\dagger}V]_k= (V^{\dagger}V,  V^{\dagger}V, \dots, V^{\dagger}V ): \bigoplus_{j=1}^{n^{k-1}}(\bigoplus_{i=1}^{n}\mathcal{H}) \to \bigoplus_{j=1}^{n^{k-1}}(\bigoplus_{i=1}^{n}\mathcal{H})$ and $[V]_k=(V,V, \dots , V): \bigoplus_{j=1}^{n^{k-1}}(\bigoplus_{i=1}^{n}\mathcal{H}) \to \bigoplus_{j=1}^{n^{k-1}}\mathcal{H}$ are both $n^{k-1}$-tuples of bounded operators on the Hilbert sapce $\bigoplus_{i=1}^{n}\mathcal{H}.$
	Observe that for $f \in \Gamma(k, I_n),$ \begin{align*}
	V_{f}(e^{(f,k)}_m)&=V_{f(1)}V_{f(2)}\dots V_{f(k)}(e^{(f,k)}_m)=V_{f(1)}V_{f(2)}\dots V_{f(k-1)}w_{f(k),m}e_{f(k)+nm}\\&=V_{f(1)}V_{f(2)}\dots V_{f(k-2)}w_{f(k),m}w_{f(k-1),{f(k)+nm}}e_{f(k-1)+n({f(k)+nm})}\\&= \vdots\\&=\prod_{i=1}^{k}w_{f(k-i),\sum_{j=1}^{i}f(k-i+j)n^{j-1}+n^{i}m}e_{\sum_{l=0}^{k-1}n^{l}f(k-l)+n^{k}m}.
	\end{align*}
	Assumimg  Equation (\ref{PPP}) and by the previous equation, we  can verify   that 
	$$\|V^k(h_f)\|^2 \leq d_m(\|[V]_kh_f\|^2-\|[V^{\dagger}V]_kh_f\|^2)+\|[V^{\dagger}V]_kh_f\|^2),$$  for all $k \in \mathbb{N}$ and  $  h_f \in\bigoplus_{f \in \Gamma(k,I_{n})}\mathcal{H}=\bigoplus_{j=1}^{n^{k-1}}(\bigoplus_{i=1}^{n}\mathcal{H}) ,$ that is, the operator $V$ satisfies the growth condition \ref{2.7}.  Hence $V$  satisfies  the hypothesis of Theorem \ref{TT}.  In particular, if $n=1$ then  the subspace $\overline{\text{span}}\{e_m\: : \: m <0\}  $  is reducing for $V$ and  the restriction of V on the subspace  $\overline{\text{span}}\{e_m\: : \: m <0\}  $ is unitary $($see \cite[Proposition  9]{EMZ15}$)$.  Note that,  the number $n$ may chosen to be $\infty$ (infinity) also.

	
\end{remark}
\paragraph{$\mathbf{Acknowledgment}$} We are thankful  to the referee for the valuable sugesstions and comments. Azad Rohilla is supported by a UGC fellowship (File No: 16-6(DEC.2017)
/2018(NET/CSIR)). Shankar Veerabathiran thanks ISI Bangalore for  Visiting Scientist position. Harsh Trivedi is supported by MATRICS-SERB  
Research Grant, File No: MTR/2021/000286, by the Science and Engineering Research Board
(SERB), Department of Science \& Technology (DST), Government of India. We acknowledge the Centre for Mathematical \& Financial Computing and the DST-FIST grant for the financial support for the computing lab facility under the scheme FIST ( File No: SR/FST/MS-I/2018/24) at the LNMIIT, Jaipur.

\end{document}